\documentclass[11pt]{article}
\usepackage{latexsym,amsmath,amsthm,amssymb}
\usepackage[sorted]{amsrefs}
\usepackage{graphicx}
\usepackage[colorlinks=true, linkcolor=blue]{hyperref}
\usepackage{geometry}
\usepackage{enumitem}
\usepackage[rgb]{xcolor}

\newtheorem{thm}{Theorem}[section]
\newtheorem{prop}[thm]{Proposition}
\newtheorem{cor}[thm]{Corollary}
\newtheorem{lem}[thm]{Lemma}
\newtheorem{conj}[thm]{Conjecture}
\newtheorem{claim}[thm]{Claim}

\theoremstyle{definition}

\theoremstyle{definition}

\theoremstyle{definition}
\newtheorem{defn}[thm]{Definition}

\theoremstyle{definition}
\newtheorem{example}[thm]{Example}

\theoremstyle{definition}

\theoremstyle{definition}

\theoremstyle{definition}
\newtheorem{fact}[thm]{Fact}

\theoremstyle{remark}

\theoremstyle{remark}
\newtheorem{remark}[thm]{Remark}

\setlist[enumerate]{topsep=0pt,partopsep=1ex,parsep=1ex}
\usepackage{dsfont}
\newcommand{\NN}{\ensuremath{\mathbb N}}

\newcommand{\RR}{\ensuremath{\mathbb R}}

\newcommand{\C}{\mathcal{C}}
\newcommand{\X}{\mathcal{X}}
\newcommand{\Z}{\mathcal{Z}}
\newcommand{\h}{\mathcal{H}}
\newcommand{\W}{\mathcal{W}}
\newcommand{\E}{\mathcal{E}}

\renewcommand{\S}{\mathcal{S}}

\renewcommand{\P}{\mathcal{P}}

\def\vc{\operatorname{vc}}

\def\esb{\mathrm{esb}}

\def\bd{\mathrm{bd}}
\def\fcl{\mathrm{fcl}}

\def\Ind#1#2{#1\setbox0=\hbox{$#1x$}\kern\wd0\hbox to 0pt{\hss$#1\mid$\hss}
\lower.9\ht0\hbox to 0pt{\hss$#1\smile$\hss}\kern\wd0}

\def\ind{\mathop{\mathpalette\Ind{}}}

\def\R{R}
\def\cR{\mathfrak{R}}

\renewcommand{\phi}{\varphi}
\renewcommand{\hat}{\widehat}
\renewcommand{\tilde}{\widetilde}
\newcommand{\tildepi}{\tilde{\pi}}

\renewcommand{\int}{\operatorname{int}}

\newcommand{\heq}{~\hat{=}~}

\def\acl{\operatorname{acl}}

\title{On the shatter function of semilinear set systems}
\author{
Abdul Basit\thanks{School of Mathematics, Monash University, Melbourne VIC 3800, Australia; \texttt{abdul.basit@monash.edu}.
Supported by Australian Research Council grant DP220102212.}
\and Chieu-Minh Tran\thanks{Department of Mathematics, National University of Singapore; \texttt{trancm@nus.edu.sg}.
}
}

\begin{document}

\maketitle

\begin{abstract}
We show that the shatter function of a semilinear set system on $\mathbb{R}^m$ is asymptotic to a polynomial. This confirms, for the structure $(\mathbb{R}; +, <)$, a conjecture of Chernikov and is a step towards characterizing model-theoretic linearity via shatter functions.
\end{abstract}

\section{Introduction}

A \emph{set system} $\S$ on a given set $V$ is a collection of subsets of $V$. The \emph{trace} of a set system $\S$ on a subset $A \subseteq V$ is $\S \cap A := \{ S \cap A: S \in \S \}$ and the \emph{shatter function} $\pi_\S : \NN \rightarrow \NN$ of $\S$ is
$$ \pi_\S(t) := \max\{ |\S \cap A|: A \subseteq V, |A| = t \}. $$
The shatter function is a natural measure of the complexity of a set system and is extensively studied in computer science, learning theory, combinatorics, logic, and other areas. See~\cites{adhms16, chernikovnotes, matouvsek1998} for a detailed introduction to combinatorial and logical aspects. 

In this paper, we are interested in set systems that arise in
geometric settings. A {\it semilinear set} in $\RR^m$ is either the solution set of a system of finitely many linear equations and inequalities, or a finite union of such sets.
A set system $\S$ on $\RR^m$ \emph{parameterized} by $Y \subseteq \RR^n$ is \emph{semilinear} if $\S = \{S_b \subseteq \RR^m : b \in Y \}$ and the sets $Y$ and $Z = \{ (a,b) \in \RR^m \times \RR^n : a \in S_b,\,b \in Y \}$ are semilinear. Our main result is the following.
\begin{thm}
\label{thm:main}
Let $\S$ be a semilinear set system on $\RR^m$ parameterized by $Y \subseteq \RR^n$. Then there exist constants $C_1, C_2 > 0$ and an integer $0 \leq s \leq n$ such that
\[ C_1 t^{s} < \pi_\S(t) < C_2 t^{s} \quad \text{for every } t \geq 1. \]
That is, $\pi_\S(t) = \Theta(t^{s})$ (as $t \rightarrow \infty$).
\end{thm}
Theorem~\ref{thm:main} presents a natural class of set systems for which the asymptotic behavior of the shatter function can be described precisely. Remarkably, obtaining such a description for many innocuous looking set systems can be extremely challenging. For example, let $\S$ be the set system consisting of pairs of points in $\RR^2$ at unit distance from one another, i.e.,  $\S = {\{\left\{ (b_1, b_2), (b'_1, b'_2) \right\} \subseteq \RR^2 : (b_1, b_2,b'_1, b'_2) \in Y\}}$ with
\[ Y = \left\{ (b_1, b_2,b'_1, b'_2) \in \RR^4 : (b_1 - b'_1)^2 + (b_2 - b'_2)^2 = 1 \right\}. \]
Arguments in~\cite{adhms16} imply that $\pi_\S(t) = \Theta(U(t))$, where $U(t)$ is the maximum number of pairs of points at unit distance from one another and the maximum is taken over all sets of $t$ points in the plane. It is known that $U(t) = \Omega(t^{1 + c_1/\log\log t})$ for a constant $c_1 > 0$ and that $U(t) = O(t^{4/3})$ (see, e.g.,~\cite{sst1984}). The \emph{Erd\H{o}s unit distance conjecture}~\cite{erdos1946} states that $U(t) = O(t^{1 + \varepsilon})$ for every $\varepsilon > 0$. We note that, if the conjecture is true, then $\pi_S(t)$ is not asymptotic to a real power function. Examples of such set systems are already known (see, e.g.,~\cite{adhms16}*{Section 4.2.2}).

Theorem~\ref{thm:main} also provides evidence for the following viewpoint: for a ``sufficiently geometric'' set system, the asymptotic growth rate of its  shatter function can be used to characterize the ``module-like''/``field-like'' (or ``linear''/``non-linear'') dividing line in model theory. In what follows, we give some background and elaborate on this perspective.

The notion of \emph{dividing line} was introduced into model theory by Shelah (see, e.g.,~\cite{shelah2021divide}). A dividing line seeks to partition the class of first-order structures into two parts, the more simple and the less
simple, in a useful manner. For instance, on the more simple side, one often has an arsenal of tools enabling a ``geometric'' understanding of the structure (e.g., a dimension theory in algebraically closed field).
A fundamental theorem in the area is the following. 
\begin{fact}[Sauer-Shelah Lemma~\cites{sauer1972, shelah1972}]
\label{fac:sauer-shelah}
Let $\S$ be a set system. Then exactly one of the following holds:
\begin{enumerate}[label = (\alph*)]
  \item  There is $D \in \NN$ such that $\pi_\S(t) = O(t^{D})$;
    \item  $\pi_\S(t) = 2^t$ for all $t>0$.
\end{enumerate}
\end{fact}
We note that, if Fact~\ref{fac:sauer-shelah}(a) holds, then $D$ can be chosen to be the \emph{VC-dimension} of the set system $\S$  (see, e.g., \cite{adhms16} for definitions and details).

The dichotomy in the Sauer-Shelah Lemma corresponds to the NIP/IP dividing line in model theory. Given a structure $M$, a set system $\S$ is definable in $M$ if it is of the form $\{S_b \subseteq M^m: b \in Y \}$ with both $Y\subseteq M^n$ and $Z=\{(a,b) \in M^m \times M^b: A \in S_b, b\in Y\}$ definable. Then, a structure $M$ is NIP if and only if every definable set systems has a polynomially bounded shatter function.

The above discussion suggests that one can try to characterize  other dividing lines in model theory by the growth rate of shatter function of a set system $\S$. A convenient measure of the growth rate is through the coarser notion of the \emph{VC-density} of $\S$, defined as 
$$\vc(\S) := \limsup_{t \rightarrow \infty} \frac{\log \pi_\S(t)}{\log t}.$$

The starting point of our current project is the study of VC-density in NIP structures by Aschenbrenner, Dolich, Haskell, Macpherson, and Starchenko~\cite{adhms16}. Here, a standard example showing that the VC~density may be rational is the semialgebraic set system $\S$ consisting of pairs of points in $\RR^2$ with unit inner product. That is, 
$\S = \{S_{b} : b \in Y\}$ where $$ Y = \left\{ (b_1, b_2,b'_1, b'_2) \in \RR^4 : b_1b'_1+ b_2b'_2 = 1 \right\}, \text{ and } S_{(b_1,b_2, b'_1, b'_2)} = \left\{ (b_1, b_2), (b'_1, b'_2) \right\}. $$
As before,  arguments in~\cite{adhms16} imply that $\pi_\S(t) = \Theta (I(t))$, where $I(t)$ is the maximum number of incidences between a set of $t$ points and a set of $t$ lines in $\RR^2$. As a consequence of the Szemer\'edi-Trotter theorem~\cite{st1983} and a corresponding lower bound construction, we have $\pi_{\S}(t) = \Theta(I(t)) = \Theta(t^{4/3})$ and, hence, that $\vc(\S) = \frac{4}{3}$. Interestingly, we are not aware of a set system consisting of pairs of points in $\RR^2$ whose VC-density is not in $\{0, 1, \frac{4}{3}, 2\}$.

In the preceding example, the set system is defined using the inner product, which involves both addition and multiplication. It seems impossible to define a set system with non-integer VC-density using only addition or only multiplication. 
This lends credence to the conjecture that the growth rate of the shatter function may allow us to distinguish between ``module-like'' structures, such as $(\RR, +, <)$ or $(\RR, \times,<)$, and ``field-like'' structures, such as $(\RR,+, \times,<)$. This ``module-like''/``field-like'' dividing line plays an important role in many classical results of model theory, such as the finite-axiomatizability problem and Mordell-Lang conjecture for function fields.

The above intuition  is formalized in (c) of the following conjecture of Chernikov.

\begin{conj}[\cite{artemconj}] \label{q:irrational}
Let $(M; \ldots)$ be an NIP structure and let $\S = \{ S_b : b \in Y \}$ with $Y \subseteq M^n$ be a definable set system. Then the following hold:
\begin{enumerate}[label = (\alph*)]
\item $\vc(\S)$ is rational;
\item if $M$ is o-minimal, there is a finite set $\Delta := \Delta(n) \subseteq [0, n] \cap \mathbb{Q}$ such that $\vc(\S)\in \Delta$; \item if $M$ is o-minimal and weakly locally modular, then $\vc(\S) \in \{0, 1, \ldots, n\}$. 
\end{enumerate}
\end{conj}
The notion of weakly locally modular defined in~\cite{berenstein2012weakly} captures the intuition that the structure is ``module-like''. In o-minimal structures, Zilber's trichotomy principle holds, asserting that an o-minimal structure that is not ``module-like'' must be ``field-like'' (interpreting a field). Hence, if an o-minimal structure is not weakly locally modular, then we can encode the earlier example to get VC-density $\frac{4}{3}$. If Conjecture~\ref{q:irrational}(c) holds, the growth rate of the shatter function indeed characterizes the ``module-like''/``field-like'' dividing line. 

We remark that, while we verify Conjecture~\ref{q:irrational} for certain structures (see Theorem~\ref{thm:mainversion2}), the phenomenon is perhaps not very straightforward. Indeed, a related conjecture of Chernikov~\cite{artemconj} concerning the logarithmic density of definable graphs has a negative answer (see Proposition~\ref{eg:rational}).

A characterization of o-minimal  expansions of $(\RR; +)$ (see Facts~\ref{fac: o-min trich}, ~\ref{fac: vectorislinear}, and~\ref{fact: Onebasedorderstructure}) together with Theorem~\ref{thm:mainversion1general}, 
a generalization of Theorem~\ref{thm:main} for ordered vector spaces over an ordered division ring,  implies Theorem~\ref{thm:mainversion2} below. 

\begin{thm} \label{thm:mainversion2}
Let $\mathfrak{R}$ be is an o-minimal structure expanding $(\RR; +, <)$. Then the following are equivalent:
\begin{enumerate}[label = (\alph*)]
    \item $\mathfrak{R}$ is weakly locally modular;
    \item for any set system $\S$ definable in $\mathfrak{R}$, $\pi_\S(t)$ is asymptotic to a polynomial.
\end{enumerate}
\end{thm}

Our result confirms Conjecture~\ref{q:irrational}(c) for expansions of $(\RR; +, <)$. Via the exponential map, one can deduce the same result for expansions of $(\RR; \times, <)$; see Corollary~\ref{cor:multversion}. Moreover, geometric stability results about weakly locally modular o-minimal structures tell us that such structures are in a Euclidean neighborhood of every point at most as complicated as an elementary extension of the structures treated in Theorem~\ref{thm:mainversion2}. This provides evidence that Conjecture~\ref{q:irrational}(c) or even more  general statements (for settings with Zilber's trichotomy and analogs of geometric stability results) are true.

We remark that resolving  Conjecture~\ref{q:irrational}(a,b) would require a deeper understanding of the VC-density of semialgebraic set systems, such as those arising from the Erd\H{o}s unit distance conjecture discussed earlier. 
An analysis similar to ours could potentially reduce Conjecture~\ref{q:irrational}(a,b) to suitable statements about incidences. 

\paragraph{Acknowledgments}

We would like to thank Artem Chernikov and Sergei Starchenko for valuable discussions during the early stages of this work.

\section{Proof overview and organization}
\label{sec:proofoverview}

The starting point of our proof is a strategy from~\cite{adhms16} which relates the shatter function of a certain set systems to the number of edges in a suitable graph.
Suppose $\S$ is a set system on $\RR^m$ all of whose members have cardinality at most two.
Let $G := G(\S) = (X, Y; E)$ be the bipartite graph with $X = Y = \RR^m$ and $E = \{(a_1, a_2) \in X \times Y : \{a_1, a_2\} \in \S \}$. For a bipartite graph $G = (X, Y; E)$ and $V \subseteq X \sqcup Y$, we set
$$ G[V] := (X \cap V, Y \cap V, E \cap (V \times V)) \quad\text{and}\quad \delta_G(t) := \max_{\substack{V \subseteq X \cup Y\\|V| = t}} |E(G[V])|.$$
A consequence of~\cite{adhms16}*{Lemma~4.1} is that $\pi_\S(t) = \Omega(\delta_G(t))$ and that $\pi_\S(t) = {O(\delta_G(t) + t)}$.
Note that if $\S$ has finitely many members, then $\pi_\S(t) = \Theta(1)$. Otherwise, if $\S$ has infinitely many members, then it is an easy result that $\delta_G(t) = \Omega(t)$. Hence, for any set system $\S$ with infinitely many members, $\pi_\S(t) = \Theta(\delta_G(t))$.

Suppose now that $\S$ is a semilinear set system. Then, by quantifier elimination for ${(\RR, <, +)}$ (see, e.g.,~\cite{vandries1998}*{Corollary 7.8}), $E$ is a semilinear set in $\RR^{2m}$. A result of Basit, Chernikov, Starchenko, Tao, and Tran~\cite{bcstt2021}  (see Fact~\ref{lem:dependent-relation}) allows us to show that $\delta_G(t)$ is asymptotic to a polynomial. This proves Theorem~\ref{thm:main} for set systems whose members consist of at most two elements.

Unfortunately, this strategy does not easily generalize, even for a set system  $\S$ whose members have carnality at most three. We can similarly construct a $3$-uniform hypergraph $G$ such that $\pi_\S(t) = \Omega(\delta_G(t))$ and $\pi_\S(t) = O(\delta_G(t) + t^2)$ (where $\delta_G$ for hypergraphs is defined similarly). It is possible that $\delta_G(t) = O(t)$, in which case we obtain $\pi_\S(t) = \Omega(t)$ and $\pi_\S(t) = O(t^2)$, with no obvious way of closing the gap.

Our approach to the proof of Theorem~\ref{thm:main} (put together in Section~\ref{sec:mainproof}) consists of two steps.
In Step~1 (Sections~\ref{sec:compindex} and~\ref{sec:simple}), we consider set systems with relatively simple descriptions, referred to as \emph{simple set systems} (see Definition~\ref{def:simple}). Roughly, these are systems whose members are unions of points and half-lines (at most 1-dimensional objects). Extending the ideas described above, we show that the shatter function of a simple set system must be asymptotic to a polynomial.
In Step~2 (Sections~\ref{sec: Uniformization} and~\ref{sec: Decomdecon}), we show that the shatter function of any given semilinear set system is asymptotic to the shatter function of a simple set system.

Below, we provide a brief overview of the key ideas involved in each step. For notational convenience and to better illustrate the definition of a simple set system, we discuss Step 2 first, followed by Step 1.

\subsection{Step 2: Uniformization, decomposition, and deconstruction}
\label{sec:proofoverview1}

The asymptotic behavior of the shatter function depends on the ``geometric shape'' of the members of the set system under consideration and any analysis must take this into account. To see this, consider the following set systems with with the same parameter set:
$$ \S_1 = \left\{ B_{a,b} \subseteq \RR^2: (a,b)\in \RR^2 \right\} \quad\text{ and }\quad \S_2 = \left\{ C_{a,b}  \subseteq \RR^2: (a, b) \in \RR^2 \right\},$$
where $B_{a,b} = {\{ (a-1,b+1), (a+1, b-1) \}}$ is a pair of points centered at $(a,b)$ and $C_{a,b} = \{ (a, y) : y \in \RR  \} \cup \{(x,b) :x \in \RR\}$  is an \emph{infinite cross} centered at $(a,b)$. Despite the set systems having the same parameter set, we have $\pi_{\S_1}(t) = \Theta(t)$ while $\pi_{\S_2}(t) = \Theta(t^2)$. This can be verified directly, though it is not entirely straightforward, even in this relatively simple case. In one direction, the inequalities $\pi_{\S_1}(t) \geq t + 1$ and $\pi_{\S_2}(t) \geq \binom{t}{\leq 2}$ follow by considering the set 
$ A = \left\{ (k, k) \in \NN^2: 1 \leq k \leq t \right\}$.

Below, we outline our indirect approach to bounding the shatter function, which is generalizable, using $\S_1$ and $\S_2$ as illustrative examples. The idea is to decompose the members of a set system into simpler objects whose shatter functions are easier to estimate. For instance, each member of $\S_1$ will be decomposed into two points and each member of $\S_2$ will be decomposed into two lines. 

In Section~\ref{sec:composite}, we introduce the notion of \emph{composite spaces}. For the examples of $\S_1$ and $\S_2$, let $V = \RR^2 \sqcup \RR^2$ be the disjoint union of two copies of $\RR^2$; we may identify $V$ with the subset $\left(\RR^2 \times \{1\}\right) \cup \left(\RR^2 \times \{2\}\right)$ of $\RR^3$ but will avoid doing so for notational simplicity. We refer to $V$ as a \emph{composite space} and take $I = \{1, 2\}$ to be the \emph{index set} of $V$. We use $I$ to refer to the components of the composite space, so the two component spaces of $V$ are $V[1] = \RR^2$ and $V[2] = \RR^2$.

We will relate the shatter function of the set system $\S_1$ and $\S_2$ to those of the following set systems in the composite space $V = \RR^2 \sqcup \RR^2$:
$$ \S'_1 = \left\{ B'_{a,b}  \subseteq V: (a,b)\in \RR^2 \right\} \quad\text{and }\quad \S'_2 = \left\{ C'_{a,b}\subseteq V: (a, b) \in \RR^2 \right\},$$
where $B'_{a,b} = \{ (a-1, b+1) \} \sqcup \{ (a+1, b-1)\}$ and  $C'_{a,b} = \{ (a, y) : y \in \RR \} \sqcup \{(x,b) :x \in \RR\}$. Here $\S_1$ is a simple family, and we remark that the shatter function of set system $\S'_2$ is the same as the shatter function of the simple family $\{ \{ a \} \sqcup \{ b \} \subseteq \RR \sqcup \RR : (a, b) \in \RR^2\}$.

Before proceeding, it will be instructive to compare $\S'_1$ to the set system  
$$ \S''_1 = \left\{ \{(a_1 - 1, b_1 + 1)\} \sqcup \{(a_2 + 1, b_2 - 1)\} \subseteq V: (a_1, b_1), (a_2, b_2) \in \RR^2 \right\}.$$
We extend, in the natural manner, the use of the index set $I$ to subsets of the composite space. 
Note that if $B' \in \S'_1$ satisfies $B'[1] = \{(a, b)\}$ for a given $(a, b) \in \RR^2$, then $B'[2] = \{(a+2, a-2)\}$, i.e., one of the coordinates determines the other one. On the other hand, for a given $(a, b) \in \RR^2$, we have $\{B'' \in \S''_2 : B''[1] = \{(a, b)\} \} \cong \RR^2$. Using this observation, it is possible to argue that $\pi_{S'_1}(t) = \Theta(t)$ while $\pi_{S''_1}(t) = \Theta(t^2)$. Hence, it makes sense to compare $\S_1$ to $\S'_1$ and not $S_1$ to $\S''_1$.

Moving forward, that $\pi_{\S_1}(t) = \Theta(t)$ and $\pi_{\S_2}(t) = \Theta(t^2)$ follow from three facts:
\begin{enumerate}[label = (\alph*)]     \item\label{itm:posimplebound} $\pi_{\S'_1}(t) = \Theta(t)$ and $\pi_{\S'_2}(t) = \Theta(t^2)$;
    \item\label{itm:poupperbound} $\pi_{\S_1}(t) \leq \pi_{\S'_1}(2t)$ and  $\pi_{\S_2}(t) \leq \pi_{\S'_2}(2t)$;
    \item\label{itm:polowerbound} $\pi_{\S_1}(t) \geq \pi_{\S'_1}(t)$ and  $\pi_{\S_2}(t) \geq \pi_{\S'_2}(t)$.
\end{enumerate}
The ideas going into the proof of~\ref{itm:posimplebound} are the subject of Step 1 outlined in Section~\ref{sec:proofoverview2}. That~\ref{itm:poupperbound} holds follows from the definitions and the observation that $|\S_i \cap A| \leq |\S_i \cap (A \sqcup A)|$ --- this will be discussed in Section~\ref{sec:composite}. 

From here on, we focus on~\ref{itm:polowerbound} and its counterpart for an arbitrary set system, arguing informally that $\S_i$ is ``more complex'' than $\S_i'$. 
First, observe that restricting to a subset of the parameter space can only decrease the shatter function.
In particular, with $U = (0,1)^2 \subseteq \RR^2$, 
$\S_2$ is ``more complex'' than the set system
$$\S_2\upharpoonright U := \left\{ C_{a,b}  \subseteq \RR^2: (a, b) \in U \right\}.$$

If we look at members of the set system $\S_2\upharpoonright U$ through the ``window'' $W_1 = (0, 1) \times (1, 2)$, we see vertical intervals, which form a set system as complex as the family of vertical lines 
$  \{ (a, y) : y \in \RR \}_{(a, b)\in \RR^2} $. Likewise, though the window $W_2 = (1, 2) \times (0,1) $, the set $\S_2\upharpoonright U$ is  as complex as the family of horizontal lines 
$  \{ (a, y) : y \in \RR \}_{(a, b)\in \RR^2} $. Thus $\S_2$ is more complex than the individual components of $\S'_2$. 

To compare $\S_2$ to $\S'_2$, we introduce the \emph{clone} (Definition~\ref{def:clones}) of the $\S_2$ given by
$$\tilde{\S}_2 = \left\{ C_{a,b} \sqcup C_{a,b}  \subseteq V: (a, b) \in \RR^2 \right\}.$$
In Section~\ref{sec:compindex}, we show the rather intuitive fact that $\pi_{\S_2}(t) = \pi_{\tilde{S}_2}(t)$. 
It now suffices to show that $\tilde{S_2}$ is more complex than $\S'_2$, which is achieved by considering $\tilde{S_2} \upharpoonright U$ and the ``combined window'' $W = W_1 \sqcup W_2$.

A similar argument works for~$S_1$, with $U = (0, 1)^2$ and $W = (-1, 0) \times (1, 2) \sqcup (1, 2) \times (-1, 0)$.

An arbitrary set system $\S$ is handled similarly, but requires a lot more work. We first use a compactness argument from model theory to reduce to the case where $\S$ is ``sufficiently uniform''. For instance, all members of $\S$ have the same dimension, 
and $\S$ does not contain both a triangle and a rectangle at the same time.
The precise definition is rather technical; see Section~\ref{sec: Uniformization} for details. 
Finally, the above restriction and window arguments, detailed in Section~\ref{sec: Decomdecon}, are applied repeatedly via induction on dimension to compare $\S$ to a simple set system.

\subsection{Step 1: Shatter functions of simple set systems}
\label{sec:proofoverview2}

Recall that the members of a simple set system consist of points and half-lines. We are able to understand their shatter functions using purely combinatorial methods.

The following example of a set system, whose members consist of three elements, elaborates on the difficulties mentioned earlier and our approach to resolve it.
Let $V = \RR \sqcup \RR \sqcup \RR^2$ be the composite space indexed by the set $I = \{1, 2, 3\}$ and let \[ \S_3 = \left\{ \{x\} \sqcup \{y\} \sqcup \{(x, y)\} \subseteq V : (x, y) \in \RR^2 \right\}. \] Considering the set $A = \{ \{k\} \sqcup \{k\} \sqcup \{(k, k)\} : k \in \NN, 1 \leq k \leq t  \}$ implies that $\pi_{\S_3}(t) \geq |\S_3 \cap A| = \Omega(t^2)$.
On the other hand, we have
\begin{equation}
\label{eq:3pointeg} \pi_{\S_3}(t) \leq \prod_{i \in I} \pi_{\S_3[i]}(t) = O(t^3).
\end{equation}
However, it can be verified directly that $\pi_{\S_3}(t) = \Theta(t^2)$. Roughly speaking, \eqref{eq:3pointeg} does not hold with equality because of the ``dependence'' of different indices on each other. Let $S \in \S_3$ and note that $S[3]$ determines $S[1]$ and $S[2]$. Equivalently, $S[1]$ and $S[3]$ together determine $S[3]$. This behavior shows up, for instance, if we consider the shatter functions of different subfamilies of $\S_3$:
\begin{enumerate}[label = (\alph*)]
    \item $\pi_{\S_3[1]}(t) = \Theta(t)$, $\pi_{\S_3[2]}(t) = \Theta(t)$, and $\pi_{\S_3[3]}(t) = \Theta(t)$;
    \item $\pi_{\S_3[\{1, 2\}]}(t) = \Theta(t^2)$;
    \item  $\pi_{\S_3[\{1, 3\}]}(t) = \Theta(t)$ and $\pi_{\S_3[\{2, 3\}]}(t) = \Theta(t)$.
\end{enumerate}

In Section~\ref{sec:indexedshatterfunction}, we introduce the \emph{forced shatter function} which captures this dependence. 
Given $J \subseteq I$ and $A \subset V$, the $J$-forced shatter function of a set system $\S$ consists of members of $\S \cap A$ that have non-empty intersection with $V[i]$ for each $i \in J$.
Restricting to such a subset makes the dependence explicit; e.g., for the set system $\S_3$, requiring that a point of $A \subseteq V$ is picked out in $V[3]$ determines the intersection pattern in $V[1]$ and $V[2]$, explaining item (c) above. In Lemma~\ref{lem:indexeddensity}, we show that the shatter function of $\S$ is asymptotic to the maximum of the forced shatter functions.

Next, we reduce the problem of bounding the forced shatter function to that of counting the number of edges in a certain semilinear hypergraph. For a simple set system whose elements consist only of points, such as $\S_3$, this is essentially the approach as described in the beginning of this section. The forced shatter function is asymptotic to $\delta_G(t)$, where $G := G(\S)$ is an $r$-partite $r$-uniform hypergraph and $\delta_G(t)$ is defined analogously to the graph case.

From the preceding two paragraphs, we obtain that the shatter function is  asymptotic to the maximum of $\delta_G$ where $G$ ranges over a certain family of semilinear hypergraphs. In particular, note that we have to consider a family of hypergraphs, instead of a single graph as in~\cite{adhms16}.
Unfortunately, this is not enough since individual $\delta_G$ may not be polynomials (see Proposition~\ref{eg:rational}). To get around this subtle issue, in Section~\ref{sec:extremalfunction}, we use the definition of these hypergraphs to show that, while the individual $\delta_G$ may be rational, the maximum must be an integer.

The definition of the hypergraph $G$ relies heavily on the description of the members of the set system. To handle the case when a simple set system contains a half-line, the arguments in Section~\ref{sec:simple} need to be significantly more involved.
The essential idea here is to reduce this situation to a finitary setting. For illustration, consider the set systems
$$ \S_4 = \left\{ U_{a,b} \subseteq \RR^2: (a,b) \in \RR^2 \right\} \quad\text{ and }\quad \S_5 = \left\{ V_{a,b} \subseteq \RR^2: (a,b) \in \RR^2 \right\},$$
where  $U_{a,b} = \{ (x, y) : x = a \text{ and } y \geq b \}$ and $V_{a,b} = \{ (x, y) : x = a+b \text{ and } y \geq 0 \}$ are half-lines.

Fix a point set $A \subseteq \RR^2$. Note that, for every $A' \in A \cap \S_4$, there exists $(x, y) \in \RR^2$ such that $(x, y) \in A'$ and $U_{x,y} \cap A = A'$. In words, the half-line $U_{x,y}$ picks out the subset $A' \subseteq A$ and has a point of $A'$ as its end point. In fact, the point $(x, y)$ determines the subset of $A$ that is picked out, and, hence, we may use $U_{x, y}$ as a witness of the fact that $A' \in A$. 

On the other hand, for a set system such as $\S_5$, this is not always possible. Indeed, fix ${t \in \NN}$ and consider the point set $A = \left\{ (k, 1) : 1 \leq k \leq t \right\}$. Then we are unable to find witness points as in the preceding paragraph. However, in this case the first coordinate determines the subset of $A$ that is picked out. Hence, we are able to ignore the second coordinate, which allows us to find a witness.

It turns out that the possibilities described above are essentially the only ones, which is made precise in the proof of Lemma~\ref{lem:cubevc2}. This leads us to the definition of a \emph{critical point} allowing us to define a semilinear hypergraph as in the finitary case (see Definitions~\ref{def:critical}~and~\ref{def:criticalgraph}). Finally, in Lemmas~\ref{lem:cubevc1}~and~\ref{lem:cubevc2}, we bound the forced shatter function in terms of the number of edges of this semilinear hypergraph.

\section{Families in composite spaces and forced shatter functions}
\label{sec:compindex}

For the setting under consideration, it will be notationally more convenient to work with \emph{parameterized families} instead of set systems. 
Given sets $V$ and $Y$ we say $\X$ is a \emph{family} in $V$ \emph{parameterized by} $Y$ if $\X$ is of the form
$$\X = (X_b)_{b \in Y} \text{ where } X_b \subseteq V \text{ for each } b \in Y.$$
The family $\X$ induces a set system $\S = \S(\X) = \{ X_b: b \in Y \}$ on $V$. Given a set $A \subseteq V$, we let $\X \cap A = \S \cap A = \{ X_b \cap A: b \in Y\}$. The \emph{shatter function} $\pi_\X$ is  
$$  \pi_\X(t) := \max\{ |\X \cap A|: A \subseteq V, |A| = t \}. $$ Hence, the notions of shatter function of a family and shatter function of a set system are linked in the obvious manner, i.e., $\pi_\X = \pi_{\S(\X)}$. 

The results in this section do not rely on the set system being semilinear. We adopt this more general setting because these results could prove useful in studying the shatter functions of definable families in other structures.

\subsection{Composite spaces and families in composite spaces}
\label{sec:composite}

Let $I$ be a finite set. A \emph{composite space} $V$ \emph{indexed} by $I$ is the disjoint union $V = \bigsqcup_{i \in I} V[i]$ where, for every $i \in I$,  $V[i]$ is a set which we refer to as a \emph{space} and whose elements we refer to as \emph{points}. We refer to $I$ as the \emph{index set} of the composite space $V$. For $J \subseteq I$, we let $V[J] := \bigsqcup_{i \in J} V[i]$, which is naturally a composite space indexed by $J$. In particular, we have $V=V[I]$.

Let $V$ be a composite space indexed by $I$. For $A \subseteq V$ and $i \in I$, we write $A[i]$ for $A \cap (V[i])$. Suppose $\X = (X_b)_{b \in Y}$ is a family in (the composite space) $V$. For $i \in I$, we set $\X[i]$ to be the (usual) family $\X[i] := (X_b[i])_{b \in Y}$ in $V[i]$. For $J \subseteq I$, we set $A[J] := \bigsqcup_{i \in J} A[j]$, and $\X[J] := \bigsqcup_{i \in J} \X[i]$. Hence, $A[J] \subseteq V[J]$ and $\X[J]$ is a family in the composite space $V[J]$ indexed by $J$. 

Suppose $V$ is a composite space indexed by $I$ and $V'$ is a composite space indexed by $I'$. It will often be the case that, for $i \in I$ and $i' \in I'$, the spaces $V[i]$ and $V'[i']$ are copies of the same set $U$.  As we do not want to make the bijection between these sets explicit, we will write
$V[i] \heq V'[i']$ to denote that $V[i]$ and $V'[i']$ are copies of the same set. Likewise, we write $a \heq b$ if $a \in V[i]$ and $b \in V'[i']$ correspond to the same element of $U$, $A \heq B$ if $A \subseteq V[i]$ and $B \subseteq V'[i']$ corresponds to the same subset of $U$, and $a \hat{\in} B$ if $a \in V[i]$, $B \subseteq V'[i']$, and $a \heq b $ for some $b \in B$.

We  collect some useful facts about  families in composite spaces.
\begin{lem}
Let $\X$ be a family in the composite space $V$ indexed by $I$.
\label{lem:vcsubfamily}
\begin{enumerate}[label=(\alph*)]
\item \label{itm:vcsubfamily} If $J \subseteq I$, then $\pi_{\X[J]}(t) \leq \pi_{\X[I]}(t)$ for every $t \geq 1$.
\item \label{itm:addingfamilies} Let $I_1, I_2 \subseteq I$ be such that $I = I_1 \cup I_2$. Then  $\pi_{\X[I]}(t) \leq \pi_{\X[I_1]}(t)\pi_{\X[I_2]}(t)$ for every $t \geq 1$.
\end{enumerate}
\end{lem}
\begin{proof}
\ref{itm:vcsubfamily} is immediate from the definitions. We now prove \ref{itm:addingfamilies}. Let $A \subseteq V$ be a set of $t$ points. Set $A_1 = A[I_1] \subseteq V[I_1]$ and $A_2 = A[I_2]  \subseteq V[I_2]$. Clearly, we have 
\[ |\X \cap A| \leq \left|\X \cap A_1\right| \left|\X \cap A_2\right| = \left|\X[I_1] \cap A_1\right| \left|\X[I_2] \cap A_2\right| \leq \pi_{\X[I_1]}(t) \pi_{X[I_2]}(t). \]
The assertion follows from the above inequality.
\end{proof}


In our proofs, we will need to perform various operations on a given family in a composite space. We now describe these operations and consider the effects of these operations on the shatter function. 

For the rest of this section, let $\X$ be a family in the composite space $V$ indexed by $I$. The simplest operation, which we refer to as \emph{cloning}, roughly speaking, replaces each index in $I$ and each space in $V$ with multiple copies of itself.
\begin{defn}[Clones]
\label{def:clones}
Let $\ell \in \NN^{\geq 1}$. The \emph{$\ell$-clone} of $V$ is the composite space $V^{(\ell)}$ indexed by $I \times [\ell]$ such that
$$V^{(\ell)}[(i, j)] \heq V[i] \text{ for every } (i, j) \in I \times [\ell].$$
The \emph{$\ell$-clone} of $\X$ is the family $\X^{(\ell)} = (X^{(\ell)}_b)_{b \in Y}$ in the composite space $V^{(\ell)}$ such that $$X^{(\ell)}_b[(i, j)]\heq X_b[i] \text{ for every } (i, j) \in I \times [\ell] \text{ and } b \in Y.$$
\end{defn}

Intuitively, the cloning operation preserves the shatter function since if a subset can be picked out in a \emph{cloned index}, then it could certainly be picked out in the original one. This intuition is formalized in the following lemma.
\begin{lem} \label{lem:Cloning}
Let $\ell \in \NN^{\geq 1}$. Then $\pi_{\X^{(\ell)}}(t) = \pi_{\X}(t)$ for every $t \geq 1$.
\end{lem}
\begin{proof}
It is immediate from Lemma~\ref{lem:vcsubfamily}\ref{itm:vcsubfamily} that $\pi_{\X}(t) \leq \pi_{\X^{(\ell)}}(t)$. We now show the other direction. Let $A \subseteq V^{(\ell)}$ be a set of $t$ points and set $A' \subseteq V$ be the point set corresponding to $\bigsqcup_{i \in I} \left(\bigcup_{j \in [\ell]} A[(i, j)]\right)$. More precisely, $A' \subseteq V$ is such that, for each $i \in I$, $$A'[i] = \bigcup_{j \in [\ell]} A'_{ij} \text{ with }A'_{ij} \heq A[(i,j)].$$
Since $|A'| \leq |A| = t$ and the shatter function is increasing in $t$, we obtain $|\X \cap A'| \leq \pi_{\X}(|A'|) \leq \pi_{\X}(t)$. Hence, the assertion of the lemma follows from $|\X^{(\ell)} \cap A| \leq |\X \cap A'|$.

Let $a, b \in Y$ be such that $X^{(\ell)}_a \cap A\neq  X^{(\ell)}_b \cap A$. Then, without loss of generality, there exists $(i_0, j_0) \in I \times [\ell]$ such that $\left(X^{(\ell)}_a \setminus X^{(\ell)}_b\right) \cap A[(i_0, j_0)]$ is non-empty. Noting that $\left(X^{(\ell)}_a \setminus X^{(\ell)}_b\right) \cap A[(i_0, j_0)]$ can be identified with a subset of $\left(X_a \setminus X_b\right) \cap A'[i_0]$, we obtain $X_a \cap A' \neq X_b \cap A'$. It follows that $|\X^{(\ell)} \cap A| \leq |\X \cap A'|$, which completes the proof.
\end{proof}


Another operation we require consists of ``augmenting'' a family in a composite space with a family whose members consist of a fixed Boolean combination of members of the original family. 
\begin{defn}[Boolean extensions]
\label{def:booleanextension}
Let $J = \{ j_1, \dots, j_\ell \} \subseteq I$ be such that $V[j]\heq V[j']$ for every $j, j' \in J$.
A family $\X' = (X'_b)_{b \in Y}$ in $V'[I']$ is a \emph{one-step Boolean extension} of $\X$ if the following hold:
\begin{enumerate}[label = (\alph*)]
  \item $I' = I \cup \{i'\}$;
    \item $V'[I] = V[I]$ and $\X'[I] = \X[I]$; 
    \item $V[i'] \heq V[j_1]$;
    \item there exists a Boolean function $\Phi : \{0, 1\}^{\ell} \rightarrow \{0, 1\}$ such that, for every $b \in Y$, $$ X'_b[i'] = \left\{ x \in V'[i']: \Phi\left( \mathds{1}_{x \hat{\in} X_b[j_1]}, \dots, \mathds{1}_{x \hat{\in} X_b[j_{\ell}]} \right) = 1\right\}. $$
\end{enumerate}
We say $\X'$ is a \emph{Boolean extension} of $\X$ if $\X'$ can be obtained from $\X$ by taking finitely many one-step Boolean extensions.
\end{defn}
\begin{lem}
\label{lem:Booleancombination}
Suppose $\X'[I'] = (X'_b)_{b \in Y}$ is a Boolean extension of $\X$. Then  there is a constant $r \in \NN^{\geq 1}$ such that
$\pi_{\X}(t) \leq \pi_{\X'}(t) \leq \pi_{\X}(r t)$ for all $t \geq 1$.
\end{lem}
\begin{proof}
It is immediate from Lemma~\ref{lem:vcsubfamily}\ref{itm:vcsubfamily} that $\pi_{\X}(t) \leq \pi_{\X'}(t)$. Towards the other direction, assume, without loss of generality, that $\X'$ is a one-step Boolean extension of $\X$. Let $A' \subseteq V'[I']$ be a set of $t$ points. Set $A \subseteq V[I]$ to be the set of points such that $A[I \setminus J] \heq A'[I \setminus J]$ and $A[j_k] \heq A'[j_k] \cup A'[i']$ for every $k \in [\ell]$. Note that $|A| \leq \ell t$ and, hence, to prove the lemma it suffices to show $|\X' \cap A'| \leq |\X \cap A|$.

Suppose $a, b \in Y$ are such that $X'_a \cap A' \neq  X'_b \cap A'$. Then there exists $i_0 \in I'$ such that $X'_a \cap A'[i_0] \neq X'_b \cap A'[i_0]$.
If $i_0 \in I \setminus J$, then clearly $X_a \cap A[i_0] \neq X_b \cap A[i_0]$ completing the proof.

Suppose now that $i_0 \in J$. Note that
$$X_a \cap A[i_0] \heq (X'_a \cap A'[i_0]) \cup \left(X'_a \cap (A'[i'] \setminus A'[i_0])\right)$$
and
$$X_b \cap A[i_0] \heq (X'_b \cap A'[i_0]) \cup \left(X'_b \cap (A'[i'] \setminus A'[i_0])\right).$$
Recalling that $X'_a \cap A'[i_0] \neq X'_b \cap A'[i_0]$, we obtain $X_a \cap A[i_0] \neq X_b \cap A[i_0]$ which completes the proof. Finally, if $i_0 = i'$, then a similar argument suffices. 
\end{proof}

 
\subsection{Forced shatter function}
\label{sec:indexedshatterfunction}

Let $\X$ be a family in the composite space $V$ indexed by $I$ parameterized by $Y$. For $J \subseteq I$ and $A \subseteq V$, we let
$$ \X \cap_J A := \{ X_b \cap A : b \in Y \text{ and }X_b[i] \cap A \neq \emptyset \text{ for each } i \in J\}. $$
The \emph{$J$-forced shatter function} of $\X$ is
$$ \tildepi_{\X, J}(t) := \max\{ |\X \cap_J A|: A \subseteq V,\, |A| = t\}. $$
Noting that $\X \cap_J A \subseteq \X \cap A$ for any $A\subseteq V$ and $J \subseteq I$, we obtain
\begin{equation}
\label{eq:jshatterupperbound}
\tildepi_{\X, J}(t) \leq \pi_{\X}(t) \text{ for every $J \subseteq I$}.
\end{equation}

The following example illustrates the behavior of the forced shatter function -- compare it to the set system $\S_3$ discussed in Section~\ref{sec:proofoverview2}. Note, in particular, that the $J$-forced shatter function is different from the shatter function of the family $\X[J]$, i.e., $\pi_{\X[J]} \neq \tildepi_{\X, J}$.
\begin{example}
Let $\X$ be the family in $\RR \sqcup \RR \sqcup \RR^3$ parametrized by $\RR^2$ defined as $$ \X = \big(  \{x\} \sqcup \{ y \} \sqcup \{ (x, y, z) \}\big)_{(x, y, z) \in \RR^3}. $$
Note first that $\pi_{\X}(t) = \Theta(t^2)$. The following are easy results.
\begin{enumerate}[label = (\alph*)]
    \item $\tildepi_{\X, \{1\}}(t) = \Theta(t^2)$, $\tildepi_{\X, \{2\}}(t) = \Theta(t^2)$, and $\tildepi_{\X, \{3\}}(t) = \Theta(t^2)$;
    \item $\tildepi_{\X, \{1,2\}}(t) = \Theta(t^2)$;
    \item $\tildepi_{\X, \{1,3\}}(t) = \Theta(t)$ and $\tildepi_{\X, \{2,3\}}(t) = \Theta(t)$;
    \item $\tildepi_{\X, \{1, 2, 3\}}(t) = \Theta(t)$. 
\end{enumerate}
\end{example}

The following lemma bounds the shatter function in terms of the forced shatter functions. It is key to the subsequent analysis.
\begin{lem}
\label{lem:indexeddensity}
Let $\X$ be a family in the composite space $V$ indexed by $I$. Then $$\pi_{\X}(t) = \Theta\left(\max_{J \subseteq I} \{ \tildepi_{\X, J}(t) \} \right).$$
\end{lem}
\begin{proof}
By \eqref{eq:jshatterupperbound}, we obtain $\max_{J \subseteq I} \{ \tildepi_{\X, J}(t) \} \leq \pi_{\X}(t)$.

To see tahe other direction, fix $A \subseteq V$ and consider the partition $\{ P_J : J \subseteq I\}$
of $\X \cap A$  where
$$ P_J = \{ X_b \cap A : b \in Y \text{ and }X_b[i] \cap A \neq \emptyset \text{ if and only if } i \in J\}. $$
By definition, we have $P_J \subseteq \X \cap_J A$, so $|P_J| \leq \tildepi_{\X, J}(t)$ and, hence, $|\X \cap A| \leq \sum_{J \subseteq I} \tildepi_{\X, J}(t)$. 
Since this is true for any choice of $A$, we obtain $\pi_\X(t) \leq \sum_{J \subseteq I} \tildepi_{\X, J}(t)$, completing the proof.
\end{proof}

\section{Simple Families}
\label{sec:simple}

Throughout this section, let $\cR= (R; <, 0, 1, +, (\lambda\cdot)_{\lambda \in D})$ be a vector spaces over an ordered division ring $D$. We view $\cR$ as a structure in the language $L = \{<, 0, 1, +, (\lambda\cdot)_{\lambda \in D}\}$. As usual, we will write $x \leq y$ for $(x<y) \vee (x=y)$. See Appendix~\ref{app:modelprelim} for standard model theoretic definitions and facts about $\cR$, including the definition of linear polynomials and semilinear sets in $\cR$. We highlight the fact that $\cR$ is a weakly locally modular o-minimal structure.

This general setting is necessary for applying the relevant model-theoretic results. Specifically, it is required to prove Theorem~\ref{thm:mainversion2} though the theorem is stated only for $R =\RR$. We need less to prove Theorem~\ref{thm:main}, where only $D =\RR$ required. Furthermore, only Proposition~\ref{prop:uniformclosure}, \ref{prop:uniformapproximation}, and \ref{prop:uniformization} rely on compactness from model theory, and these are very believable for $D =\RR$. Readers solely interested only in Theorem~\ref{thm:main} may treat these propositions as black boxes and assume $R=D=\RR$ for simplicity.

Suppose $\X = (X_b)_{b \in Y}$ is a family in $R^m$ parametrized by $Y \subseteq R^n$. We say that $\X$ is semilinear (in $\cR$) if the sets $Y$ and $Z = \{ (a,b) \in R^m \times R^n : a \in X_b,\,b \in Y \}$ are semilinear, so the set system $\S(\X) =  \{ X_b: b\in Y \} $ is semilinear. A family  $\X = (X_b)_{b \in Y}$ in a composite space $V$ indexed by $I$ is \emph{semilinear} if the composite space $V$ is a disjoint union of copies of Euclidean spaces and, for each $i \in I$, $\X[i]$ is a  semilinear family in $V[i]$.  

For the rest of this section, let $V$ be a composite space indexed by $I$ such that, for each $i \in I$, $V[i] \heq R^{m_i}$. Let
Let $\X = (X_b)_{b \in Y}$ be a family parameterized by $R^n$ in the composite space $V$. We will focus our attention to a restricted class of semilinear families defined below.

\begin{defn}[Simple families]
\label{def:simple}
We say $\X$ is a \emph{simple family} if, for every $i \in I$, there exist linear polynomials (with coefficients in $D$) in $n$ variables $f_{i,1}, \dots, f_{i,m_i}$ such that one of the following holds:
\begin{enumerate}[label = (\alph*)]
    \item $\X[i] =  \Big( \big\{ ( f_{i,1}(b), \dots, f_{i,m_i}(b) ) \big\} \subseteq R^{m_i}\Big)_{b \in R^n}$ -- we refer to $\X[i]$ as a \emph{family of points};
    \item 
     $\X[i] = \Big( \big\{ (f_{i,1}(b), \dots, f_{i,m_i-1}(b), t): t \geq f_{i,m_i}(b) \big\} \subseteq R^{m_i} \Big)_{b \in R^n}$  -- we refer to $\X[i]$ as a  \emph{family of half-lines}.
\end{enumerate}
\end{defn}

\subsection{Forced shatter function and density}

\begin{defn} (Critical parameter)
\label{def:critical} Suppose $\X$ is a simple family as in Definition~\ref{def:simple}.
Let $J, J'$ be such that $J' \subseteq J \subseteq I$ and let $A \subseteq V$. We say a parameter $b \in R^n$ is \emph{critical} with respect to $(A, J, J'
)$ if the following hold:
\begin{enumerate}[label = (\alph*)]
\item\label{itm:criticalJ} $X_b[i] \cap A \neq \emptyset$ for every $i \in J$;
\item \label{itm:criticalJ'} $\big(f_{i,1}(b), \dots, f_{i,m_i}(b)\big) \in A$ if and only if $i\in J'$. 
\end{enumerate}
\end{defn}

Let $J, J'$ be such that $J' \subseteq J \subseteq I$. We define the \emph{grid} $U_{J, J'}$ to be the Cartesian product $$ U_{J, J'} := \prod_{i \in J'} R^{m_i} \times \prod_{i \in J \setminus J'} R^{m_i - 1}. $$
By a \emph{sub-grid} we mean 
a subset $W \subseteq U_{J, J'}$ of the form $W = \prod_{i \in J} W_i$ with $W_i \subseteq R^{m_i}$ if $i \in J'$ and $W_i \subseteq R^{m_i - 1}$ if $i \in J \setminus J'$.

For $E \subseteq U_{J, J'} $, the \emph{extremal function} $\delta_E: \NN \to \NN$ is given by
$$ \delta_E(t) = \max_{W} |E \cap W| $$
where $W = \prod_{i \in J} W_i$ ranges over sub-grids of $U_{J, J'}$ with $\sum_{i \in J} |W_i|=t$. Note that, in combinatorial terms, $H= E \cap W$ is a $|J|$-uniform $|J|$-partite hypergraph and the extremal functions measures the maximum number of edges $H$ can have.

For a point $\overline{x} = (x_i : i \in J) \in U_{J, J'}$ and $\alpha \in R$, we denote by $P(\overline{x}, \alpha) \subseteq V[J]$ the point set
$$ \bigsqcup_{i \in J'} \{ x_i \} \sqcup \bigsqcup_{i \in J \setminus J'} \{ (x_i, \alpha) \}.$$

\begin{defn}
\label{def:criticalgraph}
The set $E_{J, J'} \subseteq U_{J, J'}$ consists of points $\overline{x} \in U_{J, J'}$ such that there is $\alpha \in R$ and $b \in R^n$ such that $b$ is critical with respect to $\big(P(\overline{x}, \alpha), J, J'\big)$.
\end{defn}

\begin{lem}
\label{lem:criticalrelation}
Let $J, J'$ be such that $J' \subseteq J \subseteq I$. Then $E_{J, J'}$ is semilinear.
\end{lem}
\begin{proof}
It is clear that $E_{J, J'}$ is definable in $\cR$. The assertion follows from quantifier elimination for the structure $\cR$ (see \cite{vandries1998}*{Corollary 7.8}).
\end{proof}

\begin{lem}
\label{lem:cubevc1}
Let $J, J'$ be such that $J' \subseteq J \subseteq I$. Then $\delta_{E_{J, J'}}(t) = O(\tildepi_{\X, J}(t))$.
\end{lem}
\begin{proof}
Let $W = \prod_{i \in J} W_i \subseteq U_{J, J'}$ be a sub-grid such that $\sum_{i \in J} |W_i| = t$. Consider the hypergraph $H := E_{J, J'} \cap W$. By definition, we have $|H| \leq \delta_{E_{J,J'}}(t)$. Below, we construct a set of $t$ points $A$ such that $|H| \leq |\X \cap_J A|$. Since, by definition, $|\X \cap_J A| \leq \tildepi_{\X, J}(t)$ and $W$ is an arbitrary sub-grid, this implies the assertion of the lemma.

For each $\bar{x} =  (x_i)_{i \in J}$ in $ H$, let $\alpha(\bar{x})$ be such that some $b \in R^n$ is critical with respect to $(P(\bar{x}, \alpha(\bar{x})), J, J')$. Let $i \in J \setminus J'$ and $w_i \in W_i$. We set
$$\beta(w_i) = \max\{ \alpha(\bar{x}) : \bar{x} \in H \text{ and } x_i =w_i \}.$$

Set $A = \bigsqcup_{i \in J} A[i] \subseteq V[J]$ to be the set of points defined as follows:
\begin{enumerate}[label = (\alph*)]
\item if $i \in J'$, then $A[i] = W_i$;
\item if $i \in J \setminus J'$, then $A[i] = \{ (w_i, \beta(w_i)) \in R^{m_i}: w_i \in W_i\}$.
\end{enumerate}
Note that $|A[i]| = |W_i|$, so $|A| = t$. 

Let $\overline{x} = (x_i: i \in J)$ be an edge of $H$.
If $i \in J'$, then we have $X_b[i] \cap A = x_i$. 
If $i \in J \setminus J'$, then $\beta(x_i) \geq \alpha(\bar{x})$ by definition. Recalling that, for such $i$, $X_b[i]$ is a half-line increasing in the last coordinate, we have $X_b[i] \cap A \neq \emptyset$ and, in particular, the projection of $X_b[i] \cap A$ onto the first $m_i - 1$ coordinates equals $x_i$. Hence, $X_b \cap A$ has the form $ \{ x_i : i \in J' \} \cup \{ (x_i, \beta_i) : i \in J \setminus J'\}$ for some $\beta_i \in R$. Specifically, we obtain $X_b \cap A \in \X \cap_J A$.

Suppose $\overline{x}'$ is an edge of $H$ such that $ \overline{x}' \neq \overline{x}$. Choosing a corresponding critical parameter $b' \in R^n$ for $(P(\overline{x}', \alpha), J, J')$, by the form of the intersection from the preceding paragraph, we have $X_{b'} \cap A \neq X_b \cap A$. Thus, $|H| \leq |\X \cap_J A|$, completing the proof.
\end{proof}

For a point $a = (a_1, \dots, a_m) \in R^m$, let $\hat{a} = (a_1, \dots, a_{m-1}) \in R^{m-1}$ be the projection of $a$ onto the first $m-1$ coordinates. For a set $A \subseteq R^m$, let $\hat{A} := \{\hat{a} : a \in A\} \subseteq R^{m-1}$ denote the projected point set.

\begin{lem}
\label{lem:cubevc2}
Let $J \subseteq I$. Then
$$ \tildepi_{\X, J}(t) = O\left(\max_{J' \subseteq J} \delta_{E_{J, J'}}(t)\right).$$
\end{lem}

\begin{proof}
Let $A \subseteq V$ be a set of $t$ points. For each $J' \subseteq J$, let $W_{J'} = \prod_{i \in J} W_{J', i}$ be the finite sub-grid of $U_{J, J'}$ where
$$ W_{J', i} = \begin{cases}
    A[i] & \text{ if } i \in J';\\
    \hat{A[i]} & \text{ if } i \in J \setminus J'.\\
\end{cases}$$
Let $H_{J'} = E_{J, J'} \cap W_{J'}$. Noting that $\sum_{i \in J} |W_{J', i}| \leq  t$, we have $|H_{J'}| \leq \delta_{E_{J, J'}}(t)$.
Below, we construct a  $c$-to-$1$ function $$ f : \X \cap_J A \to \bigcup_{J' \subseteq J} H_{J'}, $$ 
where $c = 3^{|J|}$.
This implies $$|\X \cap_J A| \leq \frac{1}{c}\sum_{J'\subseteq J} |H_{J'}| \leq \frac{1}{c}\sum_{J \subseteq J'} \delta_{E_{J, J'}}(t).$$
Since the above inequality is true for an arbitrary $A \subseteq V$ with $|A| = t$, the assertion of the lemma follows. In what follows, we construct the function $f$ and show that it has the required properties.

Let $S$ range over the set $\X \cap_J A$. We note that $S[i] \neq \emptyset$ for each $i \in J$. Consider the polytope $P := P(S) \subseteq R^n$ determined by the system of equations and inequalities $Q := \bigwedge_{i \in J} Q_i$, with $Q_i$ being a system of equalities and inequalities defined as follows.
\begin{enumerate}
\item Suppose $\X[i]$ is a family of points and that $S[i] = \{a\}$. Of course, by definition, $S[i]$ consists of a single point. We set
$$ Q_i := \bigwedge_{k = 1}^{m_i} \left(f_{i,k}(y) = a_k \right),$$
where $y$ is a tuple of $n$ variables.
\item Suppose $\X[i]$ is a family of half-lines and that $S[i] = \{a^1 \prec a^2 \prec \dots \prec a^\ell \} \subseteq A[i]$, with $\prec$ being the usual ordering on the $m_i$th coordinate. By the definition of simple, the first $m_i - 1$ coordinates of the $a^j$s are necessarily the same, i.e., $\hat{a}^j = \hat{a}^k$ for every $j, k \in [\ell]$.

Suppose there is a point $a^0 \in A[i]$ such that $\hat{a}^0 = \hat{a}^1$ and $a^0 \prec a^1$. If there are multiple such points, we pick the maximal point under the ordering $\prec$. In words, $a^0$ is the first point we hit if we extend a half-line in $R^{m_i}$ containing the set $S[i]$ in the other direction. Let
$$ Q_i := \bigwedge_{k = 1}^{m_i - 1} \left(f_{i,k}(y) = a^1_k \right) \wedge \left(f_{i,m_i}(y) \leq a^1_{m_i}\right) \wedge \left(f_{i,m_i}(y) \geq a^0_{m_i}\right).$$
If a point $a^0$ as in the preceding paragraph does not exist, we simply set
$$ Q_i := \bigwedge_{k = 1}^{m_i - 1} \left(f_{i,k}(y) := a^1_k \right) \wedge \left(f_{i,m_i}(y) \leq a^1_{m_i}\right).$$
\end{enumerate}

Note first that the polytope $P$ is non-empty. Indeed, since $S \in \X \cap_J A$, there exists $b \in R^m$ such that $X_b \cap_J A = S$ which implies $b \in P$.  Note also that if $b \in R^n$ is in the interior of the polytope $P$, then $X_b \cap A = S$. When $b$ is on the boundary of $P$, we either have $X_b \cap A[i] = S[i] $ or $X_b \cap A[i] = S[i] \cup \{a^0\}$ where $a_0$ is a point as described above. The latter happens when $i \in I$, $X[i]$ is a family of half-lines, and $f_{i,m_i}(b) = a_{m_i}^0$ when $a^0 \in A[i]$ is as before. 

Let $F := F(S)$ be a face of $P$ that is of smallest dimension (if there is more than one such face, we pick one arbitrarily). Note that $F$ must be a flat.

Let $J' := J'(S) \subseteq J$ be the set of indices $i \in J$ such that $\X[i]$ is a family of points, or one of the (at most) two inequalities in $Q_i$ achieves equality on $F$. Then we have the following.
\begin{claim}
\label{cl:polytopecritical}
If $b^* \in F$, then $b^*$ is critical with respect to $(A, J, J')$.
\end{claim}
\begin{proof}
Let $b^* \in F$. By definition of the polytope $P$, we have $X_{b^*} \cap A[i] \supseteq S[i] \neq \emptyset$ for all $i \in J$, so \ref{itm:criticalJ} is satisfied. If $i \in J'$ is such that $\X[i]$ is a family of points, then clearly \ref{itm:criticalJ'} is satisfied as well. Otherwise, by the definition of $J'$ above, one of the two inequalities in $Q_i$ achieves equality, which means that $\left( f_{i, 1}(b^*),\dots, f_{i, m_i}(b^*) \right) \in A$, i.e., \ref{itm:criticalJ'} is satisfied. 
\end{proof}

Let $b^* = b^*(S) \in F$ be an arbitrary point on $F$. We set 
$$ f(S) = \prod_{i \in J'} \big(f_{i, 1}(b^*), \dots, f_{i, m_i}(b^*)\big) \times \prod_{i \in J \setminus J'} \big(f_{i,1}(b^*), \dots, f_{i, m_i - 1}(b^*) \big) \in U_{J, J'}. $$
Note that the image $f(S)$ does not depend on the choice of $b^*$ and that, by Claim~\ref{cl:polytopecritical},  $f(S)$ is an edge of $E_{J, J'}$.

It remains to show that at most $c=3^{|J|}$ elements of $\X \cap_J A$ have the same image under $f$. From now on we fix $S$. Let $F := F(S)$, $b^* := b^*(S) \in F$, and $S^* := X_{b^*} \cap A$. Recall that $S^*$ might  be different from $S$. In particular, for each $i \in I$, either $S^*[i] = S[i]$ or $S^*[i] = S[i] \cup \{a^0\}$ with $a^0 \in A[i]$ is as described above.

Now let $S' \in \X \cap_J A$ be such that $f(S') =f(S)$. 
Since $f(S) = f(S')$, we have $J'(S) = J'(S') := J'$.
Now for every in $i \in J'$ and $k \in [m_i]$, we have $f_{i, k}(b^*) = f_{i, k}(b^*(S'))$. This implies that $F(S') = F$. In particular, we may assume that $b^*(S') = b^*$.

Let $b, b' \in R^n$ be such that $S = X_{b} \cap A$ and $S' = X_{b'} \cap A$. 
Note that, if $i \in J$ is such that $\X[i]$ is a family of points, then $S[i] = S'[i] = S^*[i]$. Indeed, by the definition of $Q_i$, for each $k \in \{1, \dots, m_i\}$, we have $f_{i, k}(b) = f_{i,k}(b') = f_{i,k}(b^*)$. Similarly, for every $i \in J$ such that $\X[i]$ is a family of  half-lines, for each $1 \leq k \leq m_i - 1$, we have $f_{i, k}(b) = f_{i,k}(b')= f_{i,k}(b^*)$. On the other hand, it need not be the case that $f_{i, m_i}(b) = f_{i, m_i}(b)$.

Suppose $S \neq S'$. Then there exists an index $i \in J$ such that $S[i] \neq S'[i]$. By the discussion in the previous paragraph, $\X[i]$ is a family of half-lines. We claim that $i \in J'$. Recall that $S'[i]$ and $S[i]$ is either $S^*[i]$ or $S^*[i]\setminus \{a^0\}$ for an $a^0 \in A[i]$ as described before.  
We consider the case where  $S^*[i] = S[i] \cup \{a^0\}$ and $S^*[i] = S'[i]$, with the other case being treated similarly. In the case under consideration, we have $f_{i, m_i}(b) > f_{i, m_i}(b^*) =  f_{i, m_i}(b')$.
From the definition of $Q_i(S)$, this implies that one of the inequalities in $Q_i(S)$ is satisfied with equality on $b^*$. In fact, this must be the inequality in $Q_i(S)$ of the form $\left(f_{i,m_i}(b) \geq a^0_{m_i}\right)$. It follows that $i \in J'$ as claimed.

We have shown that $S[i]=S'[i]$ for all $i \in J\setminus J'$. For $i \in J'$, either $S'[i] =S^*[i]$ or $S'[i] =S^*[i]\setminus \{a^0\}$. As a consequence, at most $2^{|J'|} \leq 2^{|J|}$ sets $S' \in \X \cap_J A$ satisfy $f(S') = f(S)$, completing the proof.
\end{proof}

\subsection{Extremal function for semilinear hypergraphs}
\label{sec:extremalfunction}

As mentioned in Section~\ref{sec:proofoverview}, our proof is complicated by the fact that there exist relations $E$ such that $\delta_E(t)$ is not asymptotic to a polynomial. Consider the following example.
\begin{prop}
\label{eg:rational}
With 
$$ E := \{ (x_1, x_2), (y_1, y_2),(z_1, z_2) \in \RR^2 \times \RR^2 \times \RR^2: x_2 = y_1 \text{ and } y_2 = z_1 \text{ and } z_2 = x_1\}, $$
we have $\delta_E(t) = \Omega(t^{3/2})$ and $\delta_E(t) = O(t^{7/4})$.
\end{prop}

\begin{proof}
To see the lower bound, for $k \in \NN$, consider the sub-grid
$$ W = [k]^2 \times [k]^2 \times [k]^2 \subseteq \RR^2 \times \RR^2 \times \RR^2. $$
Noting that each a point $(x_1, x_2) \in X$ is contained in an $k$ edges of $E \cap W$, we obtain $\delta_E(3k^2) =  k^3$, from which the lower bound follows.

For the upper bound, let $X, Y, Z \subseteq \RR^2$ and consider the grid $W = X \times Y \times Z$. Suppose that $|E \cap W| = \omega(t^{3/2})$. By adding ``dummy'' points to $X, Y, Z$ if needed, we assume that $|X| = |Y| = |Z| = t$.  For $c \in \RR$, let $X_c \subseteq X$ (resp., $Y_c \subseteq Y$) be the set of points lying on the line $x_1 = c$ (resp., $y_2 = c$). Set $C = \{c \in \RR: |X_c| > 0, |Y_c| > 0\}$.

Since the pair $(x_1, x_2), (y_1, y_2) \in X \times  Y$ is contained in an edge of $E \cap W$ only if $x_1 = y_2$, 
we have $ |E \cap W| = \sum_{c} |E \cap (X_c \times Y_c \times Z)|. $
Moreover, since a pair $(c, x_2), (y_1, c) \in X_c \times Y_c$ is contained in an edge only if $(x_2, y_1) \in Z$, we also have 
$|E \cap (X_c \times Y_c \times Z)| \leq \min \{ |X_c| |Y_c|, |Z| \}.$

Let $\C^{-} = \{c \in C: |X_c||Y_c| < t^{1/2} \}$. Now $|E \cap (X_c \times Y_c \times Z)| \leq |X_c||Y_c| < t^{1/2}$ and, trivially, $|C^{-}| \leq t$. Hence, the number of edges contributed by $c \in C^{-}$ is $O(t^{3/2})$. Set $C^{+} = C \setminus C^{-}$.

For each $c \in C^{+}$, we have $|X_c| |Y_c| \geq t^{1/2}$ and, hence, $\sum_{c \in C^{+}} \sqrt{|X_c||Y_c|} \geq |C^+| t^{1/4}$. On the other hand, by the Cauchy-Schwarz inequality,
$$ \sum_{c \in C} \sqrt{|X_c||Y_c|} \leq \sqrt{\sum_{c \in C} |X_c| \sum_{c \in C} |Y_c|} = t, $$
implying that $|C^+| \leq t^{3/4}$. Now, we certainly have $|E \cap (X_c \times Y_c \times Z)| \leq t$. Hence, the number of edges contributed by $c \in C^{+}$ is $O(t^{7/4})$, completing the proof.
\end{proof}

Despite the above example, we are able to show that the maximum, over $J' \subseteq J \subseteq I$, of $\delta_{E_{J, J'}}(t)$ is asymptotically a polynomial. We require the following result of Basit, Chernikov, Starchenko, Tao, and Tran~\cite{bcstt2021}. We refer to this paper for the relevant definitions.
\begin{fact}[\cite{bcstt2021}*{Lemma 5.5}]\label{lem:dependent-relation}
Assume that $T$ is geometric and weakly locally modular, and $\mathcal{M} = (M, \ldots) \models T $ is $\aleph_1$-saturated. Assume that $E \subseteq M^{d_1} \times \ldots \times M^{d_r}$ is an $r$-ary relation defined by a formula with parameters in a finite tuple $b$, and $E$ contains no $r$-grid $A=\prod_{i \in [r]}A_i$ with each $A_i \subseteq M^{d_i}$ infinite. Then for any $(a_1, \ldots, a_r) \in E$ there exists some $i \in [r]$ so that $a_i \in \acl \left( \left\{a_j : j \in [r] \setminus \{i\} \right\}, b\right)$.
\end{fact}

Noting that $(R; +, <)$ is geometric and weakly locally modular, and using compactness, we obtain the following.
\begin{cor}
\label{cor:zarankiewicz}
	Suppose that $E$ is a non-empty open set of an affine subspace $V$ of $R^{d_1} \times \dots \times R^{d_k}$. Assume further that there is $t>0$ such that $E$ does not contain a set of the form $A_1 \times \cdots A_k$ with $A_i \subseteq R^{d_i}$ and $|A_i|=t$. Then there is $i \in [k]$ such that the projection
 $$p_i: R^{d_1} \times \dots \times R^{d_k} \to R^{d_1} \times \dots \times R^{d_{i-1}} \times R^{d_{i+1}}\dots  \times R^{d_k} $$
 is one-to-one on $V$.
\end{cor}

We say that $E$ is \emph{independent} if there is no $i \in [k]$ such that the projection $p_i$ (as in Corollary~\ref{cor:zarankiewicz}) is one-to-one on $V$, and \emph{dependent} otherwise.

The main result of this section is the following.
\begin{thm}
\label{thm:simpleshatter}
If $\X$ is a simple family, then there is $s \in \NN$ such that $\pi_\X(t) = \Theta(t^s)$.
\end{thm}
\begin{proof}
By Lemmas~\ref{lem:indexeddensity},~\ref{lem:cubevc1}, and~\ref{lem:cubevc2}, we have 
\begin{equation}
\label{eq:integrality}
\pi_{\X}(t) = \Theta\left(\max_{J \subseteq I}\{\tildepi_{\X, J}(t)\}\right) = \Theta\left(\max_{J' \subseteq J \subseteq I}\{\delta_{E_{J, J'}}(t)\}\right).
\end{equation}
By quantifier elimination, for each $J' \subseteq J \subseteq I$, there is $N(J, J') \in \NN$ such that $E_{J, J'}$ is a disjoint union $\bigcup_{n \leq N(J, J')} E_{J, J'}^n$ where $E_{J, J'}^n \neq \emptyset$ consists of the elements of an affine space $F_{J, J'}^n$ satisfying a system of strict linear inequalities. Along with~\eqref{eq:integrality}, this implies
\begin{equation}
\label{eq:integrality2}
\pi_{\X}(t) = \Theta\left(\max_{
\begin{subarray}{c}  J' \subseteq J \subseteq I\\ n \leq N(J, J')\end{subarray}}
\{\delta_{E^n_{J, J'}}(t)\}\right).
\end{equation}

The assertion of the theorem follows from the following claim.

{\bf Claim:} If $E_{J, J'}^n$ is dependent, then $\delta_{E_{J, J'}^n}(t) = O (\delta_{E_{\hat{J}, \hat{J}'}^{\hat n}}(t)) $ for some proper subset $\hat{J} \subseteq J$, $\hat{J}' \subseteq J'\cap \hat{J}$, and $\hat{n} \in N(\hat{J}, \hat{J}')$.

We first show why the claim implies the theorem. By repeatedly applying the claim, we get $\delta_{E_{J, J'}^n}(t) = O (\delta_{E_{\hat{J}, \hat{J}'}}^{\hat n}(t)) $ with $|\hat{J}| < |J|$ and  $E_{\hat{J}, \hat{J}'}^{\hat n}(t)$ independent. Indeed, when $|\hat{J}|=0$, the relation $E_{\hat{J}, \hat{J}'}^{\hat n}(t)$ is vacuously independent. Thus equation~\eqref{eq:integrality2} can be further refined to 
\begin{equation}
\label{eq:integrality3}
\pi_{\X}(t) = \Theta\left(\max_{
\begin{subarray}{c}  J' \subseteq J \subseteq I\\ n \leq N(J, J')\\ E_{J, J'}^n \text{ is independent}
\end{subarray}}
\{\delta_{E^n_{J, J'}}(t)\}\right).
\end{equation}
Now, by Corollary~\ref{cor:zarankiewicz}, if $E^n_{J, J'}$ is independent, then $\delta_{E^n_{J, J'}}(t) = t^{|J|}$ and we obtain the desired conclusion.

It remains to prove the claim. We suppose $\delta_{E_{J, J'}^n}$ is dependent. Then there is $i \in I$ such that, with $\hat{J} = J \setminus\{i\}$ and $\hat{J}' = J' \setminus\{i\}$, the projection map
$$p_i: \prod_{j \in J'\setminus{i}} R^{m_i} \times \prod_{j \in J \setminus J'} R^{m_i - 1} \to \prod_{j \in  \hat{J}'} R^{m_i} \times \prod_{j \in \hat{J} \setminus \hat{J}'} R^{m_i - 1} $$ is injective.
Since $\delta_{E_{\hat{J}, \hat{J}'}}(t) = \max_{\hat{n} \leq N(\hat{J}, \hat{J}')} \delta_{E^{\hat n}_{\hat{J}, \hat{J}'}}(t)$, it suffices to show that
$$ \delta_{E^n_{J, J'}}(t) \leq \delta_{E_{\hat{J}, \hat{J}'}}(t). $$

Choose a set $A =  \bigsqcup_{j \in J} A[j]$ where $\sum_{j \in J } A[j] \subseteq  R^{m_j}$ when $j \in J'$ and $A[j] \subseteq  R^{m_j-1}$ when $j \in J'$ with $\sum_{j \in J}|A[j]| =t$ such that $$\left| E^n_{J, J'} \cap \prod_{j \in J} A[j] \right| \geq \delta_{E_{J, J'}}(t). $$
The projection   $p_i$ is injective on $E^n_{J, J'}$, so it is enough to verify that the image is a subset of $E_{\hat{J}, \hat{J'}} \cap \prod_{j \in \hat{J}} A[j]$. Consider $(x_j)_{j \in J}$ in $E^n_{J, J'} \cap \prod_{j \in J} A[j]$. By definition, there is $\alpha \in R$ and $b\in R^n$ such that $b$ is critical with respect to $(P((x_j)_{j \in J}, \alpha), J, J')$. It suffices to verify that $b$ is also critical with respect to $(P((x_j)_{j \in \hat{J}}, \alpha), \hat{J}, \hat{J}')$. This is true because $(P((x_j)_{j \in \hat{J}}, \alpha), \hat{J}, \hat{J}')$ is just  $(P((x_j)_{j \in J}, \alpha), J, J')$ removing the $i$th coordinate. 
\end{proof}


\section{Uniformization} \label{sec: Uniformization}

Our aim in this section is to define a notion of \emph{uniformity} for families and show that any given family can be partitioned into finitely many uniform families. We continue using the setting of the previous section.

\subsection{The flat topology and uniform closure}
\label{sec:flatclosure}

A flat in $R^m$ is a solution of a system of linear equation with coefficients in $D$, so it is definable in $\cR$.
The \emph{flat topology} on $R^m$ is defined to be the topology whose closed sets are finite unions of flats in $R^m$. The fact that this is a topology follows from the eventual stability of any decreasing sequence $Z_1 \supseteq Z_2 \supseteq \ldots$ of $\cR$-flats. Set $\fcl(X)$ to be the closure of $X \subseteq R^m$ with respect to the flat topology on $R^m$. When there are more than one structure at play, we will use $\fcl_\cR(X)$ to make precise the structure under consideration.

The following lemma links flat closure and dimension.

\begin{lem} \label{lem:affineclosure}
Suppose $X \subseteq R^m$ is semilinear in $\cR$. If a flat $Z$ satisfies $\dim(X \cap Z) = \dim(Z)$, then $Z \subseteq \fcl(X)$. Moreover, if $Z_1, \ldots, Z_\ell \subseteq R^m$ are flats such that none of them contains another and $X \subseteq \bigcup_{i=1}^\ell Z_i$, then 
$$ \fcl( X) =  \bigcup_{i=1}^\ell Z_i \text{
if and only if } \dim( X \cap Z_i  ) = \dim(Z_i) \text{ for every } i \in [\ell].$$
It follows that  $\dim( \fcl(X))=\dim(X)$.
\end{lem}
\begin{proof}

For the first assertion, note  that we in fact have the stronger inclusion $Z\subseteq \fcl(X \cap Z)$. This is the case because $\dim(X \cap Z) =\dim (Z)$ implies that $\fcl(X\cap Z)$ cannot be  a finite union of subflats of $Z$.

The backward implication of the second assertion follows from the first. We prove the forward implication by induction on $\ell$.
When $\ell = 1$, it suffices to note that if $X \subseteq \R^m$ is semilinear in $\cR$ over $B$, then so is $\fcl(X)$, and, furthermore,  $\fcl(X)$ has the same dimension as $X$. The inductive case follows from the fact that, in any topology, the closure of a union of sets is the union of the closures of the sets.
\end{proof} 

The fact below follows from quatifier elimination for $\cR$. 

\begin{lem} \label{lem:closuredefinable}
Suppose $X \subseteq R^m$ is semilinear over $B \subseteq R$. Then $\fcl(X)$ is linear over $B$. 
\end{lem}

\begin{proof}
Note that $X$ is defined by a semilinear formula over $B$, in other words, a disjunction of system of semilinear inequalities over $B$.
Since taking flat closure commutes with taking finite union, we can reduce the problem to the case where $X$ is defined by a formula of the form $\psi(x) \wedge \theta(x)$ where $\psi(x)$ is a conjunction of linear equation over $B$ and $\theta(x)$ is a conjunction of strict linear inequality over $B$. Then $X$ is an open subset of $\psi(\cR)$ and so has the same dimension as $\psi(\cR)$. By Lemma~\ref{lem:affineclosure}, $\fcl(X) =\psi(\cR)$ which yields the desired conclusion.
\end{proof}

As an immediate consequence of Lemma~\ref{lem:closuredefinable}, we obtain the following.
\begin{cor}
    \label{fact:stablyembedded}Let $\cR'=(R', \ldots)$ be an elementary extension of $\cR$. Suppose $X' \subseteq (\R')^m$ is definable in $\cR'$ over $R$ and that $X'$ is closed in the $\cR'$-flat topology. Then  $X' \cap R^m$ is closed in the $\cR$-flat topology. Moreover, every closed set in the $\cR$-flat topology is of this form.
\end{cor}

Let $X \subseteq \R^m$ be a semilinear set such that $Z = \fcl(X)$ is a flat. The \emph{boundary} (in $Z$) of $X$, denoted by $\bd(X)$, of $X$ consists of $a\in Z$ such that, for any open cube $C$ centered  at $a$, we have
$$C \cap X \neq \emptyset \quad \text{ and  } \quad C \cap (Z \setminus X) \neq \emptyset. $$
It is immediate that $\bd(X)$ is definable and, hence, semilinear. Note, in particular, that if $H \subseteq Z$ is a half flat defined the linear inequality $\lambda_1a_1 +\ldots \lambda_ma_m \ \square \ c$ (with ${\square \in \{<, \leq\}}$), then $\bd(H)$ is the $(d-1)$-dimensional flat contained in $Z$ defined by the equation $\lambda_1a_1 +\ldots + \lambda_ma_m = c$.

A family $\Z = (Z_b)_{b \in Y}$ of subsets of  $R^m$ with $Y \subseteq R^n$ is a \emph{flat family} in $\cR$ if there is a flat $Z \subseteq R^{m+n}$ such that, for every $b\in Y$,
$$ Z_b =\{ a \in R^m : (a,b) \in Z\}, $$
or, equivalently, if there is a system $\phi(x, y)$  of linear equations in $\cR$ such that $Z_b = \phi(\cR,b)$ for every ${b \in Y}$. We refer to $Z$ the {\it total flat} of $\Z$.
When $\Z$ is a flat family and $Y \subseteq R^n$ has full dimension, it is easy to see that the flat  $Z \subseteq R^{m+n}$ as above is the flat closure of the set $\{(a,b) \in R^m \times R^n: a \in Z_b, b \in Y\}$. 

For the rest of this section, let $\X = (X_b)_{b \in Y}$ be a semilinear family in $R^m$ parameterized by $Y \subseteq R^n$. 
\begin{defn}[Uniform closure]
\label{def:uniformclosure}
We say $\X$ has \emph{uniform closure (in the flat topology)} if there is an $\ell \in \NN$ and a collection $\Xi = \{ \Z_1, \dots, \Z_\ell \}$ where each $\Z_j = (Z_{j,b})_{b\in Y}$ is a flat family in $V$ such that
$$  \fcl(X_b) = \bigcup_{j \in [\ell]} Z_{j, b} \quad \text{for every $b \in Y$}. $$
Note that, if a uniform closure  exists, it is unique. Hence, we refer to the collection $\Xi$ as the \emph{uniform closure} of $\X$. When $\ell = 1$ and $\Xi = \{ \Z_1 \}$, we say that the flat family $\Z_1$ is the uniform closure of $\X$. 
\end{defn}

For $Y'\subseteq Y$, we set $\X \upharpoonright Y':=(X_b)_{b \in Y'}.$

A standard application of compactness gives us the following:
\begin{prop} \label{prop:uniformclosure}
There is a $k \in \NN$ and semilinear sets $Y_1, \ldots, Y_k$ that partition $Y$ such that, for every $i \in [k]$, the family $\X \upharpoonright Y_i$ has uniform closure.
\end{prop}
\begin{proof}
Let  $\phi(x,y)$ be a semilinear formula in $\cR$ such that, for every $b \in Y$, we have
$$ X_b = \phi(\cR,b). $$
For the next part of the proof, assume $\psi(x,y)$ is linear formula  over $R$. 
Recall that flat closure can be characterized using dimension (Lemma~\ref{lem:affineclosure}). In addition, dimension is definable and preserved under elementary extension (Facts~\ref{fact:dimension}(c) and \ref{fact:definabilitydimension}). Hence, there is a semilinear formula  $\rho_\psi(y)$ over $R$ such that, for every elementary extension $\cR'=(R', \ldots)$ of $\cR$ and $b' \in (R')^n$,
$$ \fcl_{\cR'}(\phi(\cR', b'))= \psi(\cR',b') \text{ if and only if } \cR' \models \rho_\psi(b').$$
In particular, the above holds when $\cR'=\cR$.

Now, let $\cR' =(R'; \ldots)$ be a fixed $(|R|+\aleph_0)$-saturated elementary extension of $\cR$ and $b' =( b'_1, \ldots, b'_n)$ be an arbitrary element in $R^n$.
Then $\phi(\cR', b')$ is semilinear over $R \cup \{ b_i : i \in [n] \}$, and, hence,  $\fcl_{\cR'}(\phi(\cR', b'))$ is linear over   $R \cup \{ b_i : i \in [n] \}$ by Lemma~\ref{lem:closuredefinable}.
Hence, one can choose a linear formula $\psi'(x,y)$ over $R$ (depending on $b'$ and possibly different from $\psi(x,y)$)  such that   
$$ \fcl_{\cR'}(\phi(\cR', b'))= \psi'(\cR',b'). $$
As a consequence, when $\psi(x,y)$ varies over linear formulas over $R$ with variables in $(x,y)$, the corresponding $\rho_\psi(\cR')$ forms a cover of $(R')^n$. 
Note that there are only $(|R|+\aleph_0)$-many choices of $\psi$ as described. Hence, by the assumption that $\cR'$ is $(|R|+\aleph_0)$-saturated, we can choose linear formulas $\psi_1(x,y), \ldots, \psi_k(x,y)$ over $R$ such that $(R')^n = \bigcup_{i \in [k]} \rho_{\psi_i}(\cR')$. Since each $\rho_{\psi_i}(y)$ is a formula over $R$, and $\cR$ is an elementary substructure of $\cR'$ we obtain
$$R^n = \bigcup_{i \in [k]} \rho_{\psi_i}(\cR).$$
Finally, set $Y_1= Y \cap \rho_{\psi_1}(R)$, and $Y_i = (Y \cap \rho_{\psi_i}(R)) \setminus \bigcup_{j<i} Y_j$ for $i \in [k] \setminus \{1\}$. It is easily verified that $Y_1, \ldots, Y_k$ form the desired partition.
\end{proof}

\subsection{Uniform essential approximation}
\label{sec:essentialapproxiation}

Throughout this section, let $X \subseteq \R^m$ be a semilinear set such that $Z = \fcl(X)$ is a flat with $\dim(Z) = \dim(X) = d$. The \emph{essential boundary} of $X$, denoted by $\esb(X)$, consists of $a\in Z$ such that
$$\dim_a(Z \cap X) =\dim_a (Z\setminus X) = \dim_a(Z) =d.$$
Here $\dim_a$ is the local dimension at $a$, see Appendix~\ref{app:modelprelim} for definitions.
In particular, for any open cube $C$ centererd at $a \in \esb(X) $, the sets $C \cap  X$ and $C \cap (Z\setminus X)$ are non-empty. Hence, $\esb(X)$ is a subset of $\bd(X)$. Note that, if $X$ is semilinear, then $\esb(X)$ is also semilinear.

The above definition of essential boundary is a variation of the definition in~\cite{hty20}.  We require the following result, which corresponds to Lemma~3.5 in~\cite{hty20}. Since the proof requires only minor modifications, we leave it as an exercise for the reader.
\begin{fact}
\label{fact:essentialdim}
If $E = \esb(X)$, and $a \in E$, then $$\dim_a(E)= \dim(E) =\dim(X)-1.$$ 
\end{fact}

A half-flat $H \subseteq Z$ is \emph{essential} for $X$ if there is $a \in \bd(H)$ and an open cube $C$ centered at $a$ such that 
$$X \cap C =  H \cap C.$$
We say such an $a \in Z$ \emph{witnesses} that $H$ is essential for $X$.

The following lemma establishes a relationship between essential half-flats and the essential boundary. 
\begin{lem}
\label{lem:witnesssetdim}
Suppose that the half-flat $H \subseteq Z$ is essential for $X$. Then we have the following:
\begin{enumerate}[label = (\alph*)]
    \item If $a\in Z$ witnesses that $H$ is essential for $X$, then $a \in \esb(X)$.
    \item the set of points witnessing that $H$ is essential for $X$ has dimension $d - 1$.
    \item \label{itm:witnesssetdim-c}$\bd(H) \subseteq \fcl(\esb(X))$.
\end{enumerate}
\end{lem}
\begin{remark}
The assertion of Lemma~\ref{lem:witnesssetdim}\ref{itm:witnesssetdim-c} can be strengthened: $\fcl(\esb(X))$ is the union of essential half-flats of $X$. Since we do not require this fact, we prove the weaker statement.
\end{remark}
\begin{proof}[Proof of Lemma~\ref{lem:witnesssetdim}]
Assertion (a) is immediate from the definitions.

To see (b), let $W$ be the set of points witnessing that $H$ is essential for $X$. Since $W \subseteq \bd(H)$, we have $\dim(W) \leq d-1$. Hence, it suffices to show $\dim(W) \geq d-1$. Note that $W$ is non-empty by definition. Let $a \in W$ and let $C$ be an open cube centered at $a$ such that $X \cap C = H \cap C$. Note that $C \cap \esb(X) \subseteq W$. Indeed, for any point $a' \in C \cap \esb(X)$, we may take an open cube $C' \subseteq C$ centered at $a'$. Evidently, $X \cap C' = H \cap C'$ and, hence, $a' \in W$. The assertion now follows by Fact~\ref{fact:essentialdim}, i.e., that $\dim(\esb(X)) = d - 1$.

Finally, (c) follows from (a), (b), and  Lemma~\ref{lem:affineclosure}.
\end{proof}

\begin{defn}[Essential approximation]
Let $\ell \in \NN$ and $\E = \{ H_1, \ldots, H_\ell \}$ be a collection of half-flats contained in $Z$. We say that $\E$ is an \emph{essential approximation} of $X$ if the following hold:
\begin{enumerate}[label = (\alph*)]
    \item for every $i \in [\ell]$, the half-flat $H_i$ is essential for $X$;
    \item if a half-flat $H\subseteq Z$ is essential for $X$, then $\bd(H) = \bd(H_i)$ for some $i \in [\ell]$;
    \item $\bd(H_i) \neq \bd(H_j)$ for every distinct $i, j \in [\ell]$.
\end{enumerate}
\end{defn}
Note that it is possible that an essential approximation is empty, i.e., $\E = \emptyset$. This happens, e.g., when $X$ is a flat or a flat minus a lower dimensional set.

\begin{prop}
There is an $\ell \in \NN$ and a collection of half-flats $\{ H_1, \ldots, H_\ell \}$ which is an  essential approximation of $X$. Moreover, if $X$ is semilinear over $B\subseteq R$, then we can choose $H_1, \ldots, H_\ell$ to be half-flats over $B$.
\end{prop}
\begin{proof}
By Lemma~\ref{lem:affineclosure}, and Facts~\ref{fact:dimandlocaldim} and \ref{fact:essentialdim}, we have $\dim(\fcl(\esb(X)))=d-1$. Hence, $\fcl(\esb(X))$ contains finitely many $(d-1)$-dimensional subflats flats of $Z$. Recall that the boundary of a essential half-flat of $X$ is a $(d-1)$-dimensional subflat of $Z$, which is contained in $\fcl(\esb(X))$ by Lemma~\ref{lem:witnesssetdim}(c). Moreover, for a given $(d-1)$-dimensional subflat of $Z$, there are exactly four half-flats $H \subseteq Z$ satisfying $\esb(H)=Z'$.
It follows that $X$ has finitely many essential half-flats, which implies the first assertion. The second assertion follows from Lemma~\ref{lem:closuredefinable}.
\end{proof}

From here on, let $\E = \{ H_1, \ldots, H_\ell \}$ be an essential approximation of $X$. We set $\P = \P(\E) = \{P_J : J \subseteq [\ell]\}$ to be the partition of $Z$ induced by $\E$ where, for each $J \subseteq [\ell]$, 
$$ P_J : =  \left( \bigcap_{j \in J} H_j \right) \cap \left( \bigcap_{j \in [\ell]\setminus J} Z \setminus H_j \right).  $$
Note that if $\E = \emptyset$, then $\P(\E)$ is the trivial partition, i.e., $\P(\E) = \{Z\}$.

\begin{lem}
\label{lem:essentialapproxdim}
Let $J \subseteq [\ell]$. Then the following hold
\begin{enumerate}[label = (\alph*)]
    \item \label{itm:essentialapproxdim1} either $\dim(P_J) = d$, or $\dim(P_J) \leq  d-2$;
    \item \label{itm:essentialapproxdim2} either $\dim (P_J \cap  X )\leq  d-1$, or  $\dim (P_J \setminus X ) \leq d-1$.
\end{enumerate}
\end{lem}
\begin{proof}
Assertion \ref{itm:essentialapproxdim1} follows from the fact that, if $k \in \NN$, and $\mathcal{H}$ is the intersection of $k$ half-flats of dimension $d$ with distinct boundaries, then either $\dim(\mathcal{H}) = d$ or $\dim(\mathcal{H}) \leq d - 2$. We prove this by induction. The assertion is clearly true when $k = 1$. Let $\mathcal{H}$ be the intersection of half-flats $h_1, \dots, h_k$ with $k > 1$, and assume that the statement is true for smaller values of $k$. Suppose, for contradiction, that $\dim(\mathcal{H}) = d-1$, and note that we must have $\dim(\bigcap_{i \in [k - 1]}h_i) = d$, otherwise the inductive hypothesis is contradicted. But then $\dim(\mathcal{H}) = d-1$ contradicts the assumption that the half-flats have distinct boundaries.

We now prove \ref{itm:essentialapproxdim2}. 
Fix $J \subseteq [\ell]$, and suppose that $\dim(X \cap P_J) = \dim ( (Z \setminus X) \cap P_J ) = d$. Then, there is a point $a \in \esb(X) \cap P_J$, a half-flat $H \subseteq Z$ containing $a$, and an open cube $C$ centered at $a$ such that $X \cap C = H \cap C$. That is, $H$ is essential for $X$. But then, by definition of $P_J$, we have $P_J \subseteq H$ or $P_J \subseteq Z \setminus H$, a contradiction.
\end{proof}

A family $\h = (H_b)_{b \in Y}$ of subsets of  $R^m$ with $Y \subseteq R^n$ is a \emph{half-flat family} in $\cR$ if there is a half-flat $H \subseteq R^{m+n}$ such that, for every $b\in Y$,
$$ H_b =\{ a \in R^m : (a,b) \in H\}. $$
We refer to $H$ the {\it total half-flat} of $\h$.
Suppose $\h$ is a half-flat family and $H \subseteq R^{m+n}$ is as above. Let $Z = \fcl(H)$ and $\Z = (Z_b)_{b \in Y}$ with $Z_b = \{a \in R^m : (a, b) \in Z\}$, i.e., $Z$ is the total flat of $\Z$. Then, for every $b \in Y$, we have $\fcl(H_b) = Z_b$.

For the rest of this section, let $\X = (X_b)_{b\in Y}$ be a semilinear family in $R^m$ parameterized by $Y \subseteq R^n$. Suppose that the uniform closure is the flat family $\Z = (Z_{b})_{b \in Y}$.

\begin{defn}[Uniform essential approximation]
A (possibly empty) collection of half-flat families $\Upsilon = \{ \h_1, \dots, \h_\ell \}$ with $\h_j = (H_{j,b})_{b\in Y}$ is a \emph{uniform essential approximation} of $\X$ if, for every $b\in Y$:
\begin{enumerate}[label = (\alph*)]
    \item the collection $\E_b := \{ H_{j,b} : j \in [\ell] \}$ is an essential approximation of $X_b$;
   \item if $\P(\E_b) = \{P_{J, b} : J \subseteq [\ell] \}$ is the partition of $Z_b$ induced by $\E_b$,
   then, for each $J \subseteq [\ell]$, the dimension of $P_{J,b}\cap X_b$ and $P_{J,b} \setminus X$ is constant as $b$ ranges over $Y$.
\end{enumerate}
Note that it is possible that a uniform essential approximation is empty. This happens, e.g., when every member of $\X$ is a flat, or a flat minus a lower dimensional set. 
\end{defn}

The main point of this section is the following proposition.
\begin{prop}
\label{prop:uniformapproximation}
There is $k \in \NN$ and semilinear sets $Y_1, \ldots, Y_k$ that partition $Y$ such that, for every $i \in [k]$, the family $\X \upharpoonright Y_i$ has a uniform essential approximation. 
\end{prop}
\begin{proof}
Let  $\phi(x,y)$ be a semilinear formula over $R$ and let $\psi(x,y)$ be the system of linear equations over $R$ such that, for every $b \in Y$, we have
$$ X_b = \phi(\cR,b) \quad \text{and} \quad  Z_b = \psi(\cR,b). $$

For the next part of the proof, assume $\Psi=  (  \psi_{1}(x,y), \ldots, \psi_\ell(x,y))$ is a  finite tuple of semilinear formula over $R$ with variables in $(x,y)$ such that, for $i \in [\ell]$,
$$ \psi_i(x,y) = \psi(x,y) \wedge \eta_i(x,y)$$
where $\eta_i(x,y)$ is a linear inequality over $R$ with variables in $(x,y)$ and $H_{i,b} := \psi_i(\cR, b)$ is a half-flat of $Z_b$. For $\varepsilon \in R$ and $a =(a_1, \ldots, a_m)\in R^m$, let $C_\varepsilon =\{ a' =(a'_1, \ldots a'_m \in R^m \mid  a_i-\varepsilon < a'_i < a_i+\varepsilon\} $.
Note that, give $X \subseteq R^m$ and $a \in R^m$, we have $a \in \esb(X)$ if and only if for all $\varepsilon \in R$, 
$$ \dim( X \cap C_{a, \varepsilon}) =\dim ( (Z\setminus X) \cap C_{a, \varepsilon} = \dim Z.   $$
Hence, whether half-flats $H_1, \ldots, H_\ell$ form an essential approximation of $X$ can also be characterized in term of dimension. 
In addition, dimension is definable and preserved under elementary extension (see Facts~\ref{fact:dimension}-(c) and \ref{fact:definabilitydimension}). Hence, we obtain a semilinear formula  $\rho_\Psi(y)$ over $R$ such that, for every elementary extension $\cR'=(R', \ldots)$ of $\cR$ and $b' \in (R')^n$, we have $\cR' \models \rho_\Psi(b')$ if and only if
$$  \{\psi_1(\cR',b'), \ldots, \psi_\ell(\cR',b')\} \text{ form an essential approximation of } \phi(\cR', b').$$

Choose finitely many system of linear equations $\psi_{b', 1}(x,y), \ldots, \psi_{b', \ell'}(x,y)$ such that   
$$ \fcl_{\cR'}(\phi(\cR', b'))= \bigcup_{i \in [\ell'] }\psi_{b', i}(\cR',b'). $$
Therefore, as $\Psi$ varies, $\rho_\Psi(\cR')$ form a cover of $(R')^n$. 
Note that there are only $(|R|+\aleph_0)$-many choices $\Psi$ as described. Hence, by the assumption that $\cR'$ is $(|R|+\aleph_0)$-saturated, we can choose $\Psi_1, \ldots, \Psi_k$. such that $(R')^n = \bigcup_{i \in [k]} \rho_{\Psi_i}(\cR')$. Recall that each $\rho_{\Psi_i}(y)$ is a formula over $R$ and $\cR$ is an elementary substructure of $\cR'$, we also have
$$R^n = \bigcup_{i \in [k]} \rho_{\Psi_i}(\cR)$$
Finally, set $Y_1= Y \cap \theta_{C_1}(R)$ and $Y_i = (Y \cap \theta_{C_1}(R)) \setminus \bigcup_{j<i} Y_j$ for $i \in [k] \setminus \{1\}$. It is easily verified that $Y_1, \ldots, Y_k$ form the desired partition.
\end{proof}

\subsection{Uniform families}
\label{sec:uniform}

Let $\X = (X_b)_{b \in Y}$ be a semilinear family in $\R^m$ parameterized by $Y \subseteq \R^n$. We define a uniform family inductively as follows.
\begin{defn}[Uniform family]
We say $\X$ is \emph{uniform} if one of the following holds:
\begin{enumerate}[label = (\alph*)]
\item $X_b = \emptyset$ for every $b \in Y$ (i.e., $\X$ is the constant family whose members are the empty set) --- in other words, $\dim(X_b) = -\infty$ for every $b \in Y$;
\item $\X$ has uniform closure given by the flat family $\Z = (Z_{b})_{b\in Y}$ and uniform essential approximation $\Upsilon = \{ \h_{j} : j \in [\ell] \}$ with $\h_{j} = (H_{j,b})_{b\in Y}$.   Then, with $\E_b := \{ H_{j,b} : j \in [\ell] \}$ being the essential approximation of $X_b$, and $\P_b = \P(\E_b) = \{P_{J, b} : J \subseteq [\ell] \}$ the partition of $Z_b$ induced by $\E_b$, one of the following holds for each $J \subseteq [\ell]$:
\begin{enumerate}[label = (\roman*)]
    \item $\dim(P_{J,b} \cap X_b) < \dim(X_b)$ for all $b \in Y$, and the family $(P_{J,b} \cap X_b)_{b \in Y}$ is uniform;
    \item $\dim(P_{J,b} \setminus X_b ) < \dim(X_b)$ for all $b\in Y$, and the family $(P_{J,b} \setminus X_b)_{b \in Y}$ is uniform.
\end{enumerate}
Note that if $\Upsilon = \emptyset$, then $\X$ is uniform.
\item $\X$ has uniform closure given by $\{\Z_1, \dots, \Z_{\ell}\}$ with $\ell > 1$ and $\Z_i = (Z_{i,b})_{b\in Y}$. Then, for each $i \in [\ell]$, the family $(X_b \cap Z_{i,b})_{b\in Y}$ is uniform.
\end{enumerate}
\end{defn}

\begin{prop}
\label{prop:uniformization}
There is a $k \in \NN$ and semilinear sets $Y_1, \ldots, Y_k$ that partition $Y$ such that, for every $i \in [k]$, the family $\X \upharpoonright Y_i$ is uniform. 
\end{prop}
\begin{proof}
If, for all $b\in Y$, we have $X_b =\emptyset$, (equivalently, $\dim(X_b) =-\infty$), then the statement is immediate. Toward using induction, assume we have proven the statement for all families $\X$ such that $\dim(X_b)\leq d$ for all $b\in Y$ with $d \in \NN \cup\{-\infty\}$. 

Now consider the case where $\X$ is a family of semilinear sets of dimension $\leq d+1$, $\X$ has uniform closure a single flat family $\Z = (Z_{i,b})_{b\in Y}$, and $X$ has a uniform essential approximation.  Note that all members of $X$ has the same dimension as a set has the same dimension as its flat closure [ref], and  all members of a flat family has the same dimension by [ref]. We can assume that $\dim(X_b) = d+1$ as otherwise the desired conclusion is immediate from the induction hypothesis. 
 Let $\Upsilon = \{ \h_{j} : j \in [\ell] \}$ be the uniform essential approximation of $\X$ with $\h_{j} = (H_{j,b})_{b\in Y}$.
Let $P_{J,b}$ be as stated. By Lemma~\ref{lem:essentialapproxdim}, for $J \in [\ell]$ and each $b \in Y$, 
$$\text{either} \quad \dim(P_{J,b}\cap X_b)< \dim(X_b) \quad \text{or} \quad \dim(P_{J,b}\setminus X_b)< \dim(X_b).$$
Using Fact~\ref{fact:dimension}, for given $J$, the sets $\{ b\in Y : \dim(P_{J,b}\cap X_b)< \dim(X_b) \}$ and $\{ b\in Y : \dim(P_{J,b}\setminus X_b)< \dim(X_b) \}$ are both definable. As there are at most $2^{\ell}$-many $J \subseteq [\ell]$, we can partition $Y$ into finitely-many semilinear sets $Y_1, \ldots, Y_{k'}$ such that for each $i \in [k']$, for all $J\subseteq [\ell]$, either $\dim(P_{J,b}\cap X_b)< \dim(X_b)$ for all $b\in Y_i$ or $\dim(P_{J,b}\setminus X_b)< \dim(X_b)$ for each $b\in Y_i$. For each $i \in [k']$, using the induction hypothesis, we can partition $Y_i$ further into finitely many semi-linear sets $Y_{i,1}, \ldots, Y_{i, m_i}$ such that for each $h \in [m_i]$, either (i) or (ii) hold with $Y$ replaced by $Y_{i,h}$. Thus, we arrived as the desired conclusion.

By using Proposition~\ref{prop:uniformapproximation} and partitioning the parameter set $Y$ into finitely many semilinear set if necessary, we can handle the case where $\X$ is a family of semilinear sets of dimension $d$, $\X$ has uniform closure a single flat family $\Z = (Z_{i,b})_{b\in Y}$. 

Next consider the case where $\X$ has uniform closure where $X$ has uniform closure $\Xi = \{\Z_1, \dots, \Z_{k_1}\}$ with $\Z_i = (Z_{i,b})_{b\in Y}$ for $i \in k_1$. Assuming we can prove the statement of the theorem for each $( X_b \cap Z_i)_{b\in Y}$ for each $i \in k_1$, the partition of $Y$ refining the resulting partition for each $( X_b \cap Z_i)_{b\in Y}$ yields the desired conclusion for $\X$.

To finish the inductive argument, it remains to consider the case where $\X = (X_b)_{b\in Y}$ is an arbitrary semilinear family with $\dim(X_b) \leq d+1$ for all $b\in Y$. This is handled by partitioning $Y$ into finitely many semilinear set such that the restricted families to these sets has uniform closure. 
\end{proof}

\section{Decomposition and deconstruction} \label{sec: Decomdecon}

Throughout this section, $\X = (X_b)_{b \in Y}$ is a semilinear family in the composite space $V$ indexed by $I$. Moreover, we suppose that the parameter space $Y$ is a semilinear open cell and that, for every $i \in I$, the family $\X[i]$ is uniform.

\subsection{Decomposition}
\label{sec:decomposition}

Throughout this section, fix $i_0 \in I$ and let $\{\Z_{j}: j \in [\ell] \}$ with $\Z_{j} = (Z_{j, b})_{b \in Y}$ be the uniform closure of $\X[i_0]$. Suppose, without loss of generality, that, for each $b \in Y$,
$$ \dim(Z_{1, b}) \geq \dim(Z_{2, b}) \geq \ldots \geq \dim(Z_{\ell, b}). $$
Let $\ell' \in [\ell - 1]$ be the smallest integer such that, for each $b \in Y$, $\dim(Z_{\ell'+1, b}) < \dim(Z_{\ell', b})$. If no such $\ell'$ exists, then let $\ell' = \ell$. 

Given a family of semilinear sets, the following definition provides a new family in a suitable composite space by ``picking out the components of highest dimension''. Take the example $(X_b)_{b\in \RR^2}$ where $X_b$ is the union of the vertical and horizontal lines passing through $b$ and the four points with distance 1 from $b$ in the diagonal directions. Its decomposition will be $(V_b \sqcup H_b\sqcup Q_b)_{b\in \RR^2}$ in the composite space $\RR^2 \sqcup \RR^2 \sqcup \RR^2$ where $V_b$ and $H_b$ are the vertical and horizontal lines passing through $b$, and $Q_b =  X_b\setminus (H_b \cup V_b)$ is the set of four points. 

\begin{defn}[Decomposition]
\label{def:decomposition}
We say that $(I', V', \X')$ with $\X' =(X'_b)_{b\in Y}$ is a \emph{decomposition of $(I, V, \X)$ at index $i_0$} if the following hold:
\begin{enumerate}[label = (\alph*)]
    \item $I' = (I \setminus \{i_0\}) \bigcup \left(\{i_0\} \times [\ell' + 1]\right)$;
    \item $V'[I \setminus \{i_0\}] = V[I \setminus \{i_0\}]$ and $\X'[I \setminus \{i_0\}] = \X[I \setminus \{i_0\}]$;
    \item $V'[(i_0,j)] \heq V[i_0]$ for every $j \in [\ell' + 1]$;
    \item $X'_b[(i_0,j)] \heq  X_b[i_0] \cap Z_{j, b}$ for every  $j \in [\ell]$ and $b \in Y$;
    \item $X'_b[(i_0,\ell' + 1)] \heq  X_b[i_0] \setminus \left(\bigcup_{j \in [\ell']} Z_{j, b}\right)$ for every $b \in Y$.
\end{enumerate}
\end{defn}
\begin{remark}
If $\ell' = \ell$, then $\X'[(i_0, \ell' + 1)]$ is the constant family whose members are the empty set. 
\end{remark}

The main point of this section is the following.
\begin{prop} \label{prop:decomposition}
Let $(I', V', \X')$ be a decomposition of $(I, V, \X)$ at index $i_0\in I$. Then the following hold:
\begin{enumerate}[label = (\alph*)]
\item\label{itm:decompositiona} there is an $r_1 \in \NN$ such that $\pi_\X(t) \leq \pi_{\X'}(r_1t)$ for every $t \geq 1$;
\item\label{itm:decompositionb} there is an $r_2 \in \NN$ and a semilinear open cell $U \subseteq Y$ such that $\pi_{\X'\upharpoonright U}(t) \leq \pi_{\X\upharpoonright U}(r_2 t)$ for every $t \geq 1$.
\end{enumerate}
\end{prop}

Before giving the proof, we collect some lemmas that help simplify the presentation.

\begin{lem} \label{lem:extensionbyclosure}
Suppose $\X$ is a family in the composite space $V[I]$. Let $i_0 \in I$ be such that such that $\X[i_0]$ has uniform closure given by the flat family $\Z$. Let $\X'[I]$ be the family such that $\X'[I\setminus\{i_0\}] = \X[I\setminus\{i_0\}]$ and $\X'[i_0] = \Z$. Then $\pi_{\X'}(t) \leq \pi_{\X}(t)$ for every $t \geq 1$.
\end{lem}
\begin{proof}
Let $A \subset V[I]$ be a set of $t$ points. We will construct $A'\subseteq V$ with $|A'|=t$ such that $|\X' \cap A| \leq |\X \cap A'|$, which implies the assertion. Since $\X[I \setminus \{i_0\}] = \X'[I \setminus \{i_0\}]$, we may assume that $A \subseteq V[i_0]$.

Note that, for any $a, b \in Y$ such that $Z_a \cap A \neq Z_b \cap A$, the flats $Z_a$ and $Z_b$ are parallel and, hence, $(Z_a \cap A) \cap (Z_b \cap A) = \emptyset$.
Furthermore, since $Z_b = \fcl(X_b)$ for each $b \in Y$, by appropriately scaling and translating the points of $A$, we may assume, that for each $A' \in \X' \cap A$, there exists $b \in Y$ such that $\X_b \cap A = A'$. This implies $|\X' \cap A| \leq |\X \cap A'|$, finishing the proof.
\end{proof}

Let $X \subseteq R^m$ and $Z \subseteq R^m$ be a flat. We denote by $\int_Z(X)$ the set of points $a \in X \cap Z$ such that there is an open cube $C$ centered $a$ with $C \cap Z \subseteq X \cap Z$.
\begin{lem} \label{lem: singleseparation1}
Suppose $X \subseteq R^m$ is semilinear and that $ Z \subseteq \fcl(X)$ is a flat with $\dim(Z) = \dim (X)$. For $k \in \NN$, let $W_1, \ldots, W_k$ be flats in $R^m$ such that $\dim (W_i \cap Z) < \dim(Z)$ for every $i \in [k]$.

Then there is an open cube $C \subseteq R^m$ such that 
$$ C \cap X  = C \cap Z \neq \emptyset\quad\text{
and }\quad C \cap W_i = \emptyset  \text{ for every } i \in [k].$$
\end{lem}
\begin{proof}
We can assume the flats contained in $\fcl(X)$ other than $Z$ are among the $W_i$'s. Note that $X \cap(W_1 \cup \ldots \cup W_k)$ has dimension less than $\dim(Z)$ which implies, by Lemma~\ref{lem:affineclosure}, that $X \setminus (W_1 \cup \ldots \cup W_k)$ is dense (with respect to the flat topology) in $Z$. In particular, we have $\dim(\int_Z(X \setminus (W_1 \cup \ldots \cup W_k))) = \dim(X \setminus (W_1 \cup \ldots \cup W_k))$.
Let $a\in \int_{Z}(X \setminus (W_1 \cup \ldots \cup W_k))$ and let $C_0$ be an open cube centered at $a$ such that $C_0 \cap (X \setminus (W_1 \cup \ldots \cup W_k))$ has the same dimension as $C_0 \cap Z $. For every $i \in [k]$, let $C_i$ be an open cube centered at $a$ such that $C_i \cap W_i = \emptyset$. 

Set $C = \bigcap_{i=0}^k C_i$. It is straightforward to verify that $C$ satisfies the assertions of the lemma.
\end{proof}

We also require a version of Lemma~\ref{lem: singleseparation1} for families.
\begin{lem}[First Window Lemma] \label{lem:familyseparation1}
Let $Y \subseteq \R^n$ be an open set and $\X = (X_b)_{b \in Y}$ be a semilinear family in $R^m$ with uniform closure. Suppose $\Z = (Z_b)_{b\in Y}$ is a flat family such that, for every $b \in Y$, $Z_b  \subseteq \fcl_\cR(X_b)$ and $\dim (Z_b) = \dim (X_b)$. Let $k \in \NN$ and, for every $i \in [k]$, let $\W_i = (W_{i,b})_{b\in Y}$ be a flat family such that $\dim (W_{i,b} \cap Z_b) < \dim(Z_b)$ for every $b\in Y$.

Then there is an open cube $C \subseteq R^m$ and a semilinear open cell $U \subseteq Y$ such that, for every $b \in U$,
$$ C \cap X_{b}  = C \cap Z_{b} \neq \emptyset \quad\text{ and }\quad C \cap W_{i,b}  = \emptyset \text{ for every } i \in [k].$$
\end{lem}
\begin{proof}
Let $X = \{ (a,b) \in R^{m+n} : a \in X_b, b \in Y\}$, let $Z$ be the total flat of $\Z$, and, for every $i \in [k]$, let $W_i$ be the total flat of $\W_i$. Using the fact that $\X$ has uniform closure, we arrange that for each $b\in Y$, all the flat contained in $\fcl(X_b)$ other than $Z_b$ are among the $W_{i,b}$'s. Note that all members of $\Z$ have the same dimension, which we denote by $d$. By assumption, the dimension of each member of $\X$ is also $d$. It follows, from Fact~\ref{fact:dimension}-(c), that $\dim(X) = d + \dim(Y) = d+n$. Similarly, we have $\dim(Z) = d+n$, and, hence, $\dim(Z) = \dim(X)$. Hence, by Lemma~\ref{lem:affineclosure}, we obtain $Z\subseteq \fcl(X)$. Similarly, we also have $\dim (Z \cap W_i) < \dim(Z)$ for every $i \in [k]$.

By Lemma~\ref{lem: singleseparation1}, there is an open cube $C \times C' \subseteq R^m \times R^n$ such that
$$(C \times C') \cap X = (C \times C') \cap Z \neq \emptyset \quad\text{ and }\quad (C \times C') \cap W_i = \emptyset\text{ for every } i \in [k]. $$
By dimension considerations along with Fact~\ref{fact:dimension}-(c), we obtain that the set $\{b \in C' \cap Y : (C \cap Z_b) \neq \emptyset \}$ has dimension $n$. Hence, by Facts~\ref{fact:semilinearfiniteuntion}~and~\ref{fact:dimension}-(d), there is a semilinear open cell $U \subseteq C' \cap Y$ such that  $C \cap Z_b \neq \emptyset$ for every $b \in Y'$. 
It is straightforward to verify that $C$ and $U$ satisfy the assertions of the lemma.
\end{proof}

\begin{proof}[Proof of Proposition~\ref{prop:decomposition}]
Suppose $(I', V', \X')$ is a decomposition of $(I, V, \X)$ at index $i_0$.
By definition, for every $b \in Y$, we can write $X_b[i_0] \heq \bigcup_{j \in [\ell' + 1]} X'_b[(i_0, j)]$. Hence, $\X$ may be identified with a family obtained by restricting a Boolean extension of $\X'$ to an appropriate index set. In particular, that \ref{itm:decompositiona} holds follows immediately from Lemmas~\ref{lem:vcsubfamily}\ref{itm:vcsubfamily} and \ref{lem:Booleancombination}. 

To prove \ref{itm:decompositionb}, we construct a family a family $\X_{\ell'}$ in the composite space $I_{\ell'}$ parametrized by $Y_{\ell'} \subseteq Y$ where $I_{\ell'} = I \cup \{h_1, h_2, \dots, h_{\ell'}\}$, $\X_{\ell'}[I] = \X[I]$, and $\X[h_k] = \Z_k$ for each $k \in [\ell']$. Moreover, we will show that there is $r \in \NN$ such that 
\begin{equation}
\label{eq:decompositiont}
\pi_{\X_{\ell'} \upharpoonright Y_{\ell'}}(t) \leq \pi_{\X}(r t) \text{ for all }t \geq 1.
\end{equation}
Observe that $\X'$ can be obtained from $\X_{\ell'}$ via a Boolean extension, specifically, by considering intersections of members of $\X[i]$ with $\Z[h_i]$ for each $i \in [\ell]$. Hence, that \ref{itm:decompositionb} holds follows from \eqref{eq:decompositiont} along with Lemmas~\ref{lem:vcsubfamily}\ref{itm:vcsubfamily} and \ref{lem:Booleancombination}.

We construct $\X_{\ell'}$ from $\X$ in a sequence of steps. We show that at each step the shatter function can only decrease possibly at the cost of restricting to a subset of the parameter space. Set $I_0 = I$, $V_0 = V$, $\X_0 = \X$, and $Y_0 = Y$. In steps $k = 1, \dots, \ell'$, given the family $\X_{k-1} = (X_{k-1, b})_{b \in Y_{k-1}}$ in $V_{k-1}[I_{k-1}]$, we obtain $\X_k = (X_{k, b})_{b \in Y_{k}}$ in $V_k[I_k]$ in a sequence of sub-steps as follows.

By Lemma~\ref{lem:familyseparation1}, there is an open cube $C_k \subseteq V_{k-1}[i_0]$ and a semilinear open cell $Y_k \subseteq Y_{k-1}$ such that, for every $b \in Y_k$,
$$ C_k \cap X_{k-1, b}[i_0]  = C_k \cap Z_{k, b} \neq \emptyset \quad\text{ and }\quad C_k \cap Z_{j,b}  = \emptyset\text{ for each } j \in [\ell] \setminus \{k\}.$$
Set $\C_k = (C_k)_{b \in Y_k}$, i.e., $\C_k$ is the semilinear family parameterized by $Y_k$ all of whose members are equal to $C_k$. Let $I_{k,1} = I_{k-1} \cup \{c_k\}$ where $c_k$ is any constant not already in $I_{k-1}$. We let $\X_{k,1}$ be such that $\X_{k, 1}[I_{k-1}] = \X_{k-1}[I_{k-1}]$ and $\X_{k,1}[c_k] = \C_k$.
By the definition of the shatter function and Lemma~\ref{lem:vcsubfamily}\ref{itm:addingfamilies}, we obtain
\begin{equation}
\label{eq:decomposition1}
\pi_{\X_{k,1}\upharpoonright Y_k}(t) = \pi_{\X_{k-1}\upharpoonright Y_k}(t) \leq \pi_{\X_{k-1}\upharpoonright Y_{k-1}}(t) \text{ for all }t \geq 1.
\end{equation}
Next, let $I_{k,2} = I_{k_1} \cup \{h_k\}$ where $h_k$ is any constant not already in $I_{k,1}$, and let $\X_{k,2}$ be such that $\X_{k,2}[I_{k,1}] = \X_{k,1}[I_{k,1}]$ and  ${\X_{k,2}[h_k] \heq (X_{k-1, b}[i_0] \cap C_k)_{b \in Y_k}}$. By Lemma~\ref{lem:Booleancombination}, there is $r_k \in \NN$ such that
\begin{equation}
\label{eq:decomposition2}
\pi_{\X_{k,2} \upharpoonright Y_k}(t) \leq \pi_{\X_{k,1} \upharpoonright Y_k}(r_k t) \text{ for all }t \geq 1.
\end{equation} 
Next, let $I_{k,3} = I_{k,2}$, and $\X_{k,3}$ be such that $\X_{k,2}[I_{k,3} \setminus \{h_k\}] = \X_{k,1}[I_{k,1}\setminus \{h_k\}]$ and $\X_{k,2}[h_k] = (Z_{k, b})_{b \in Y_k}$.
Noting that, for each $b \in Y_k$, we have $\fcl(X_{k-1, b}[i_0] \cap C_k) = Z_{k, b}$, by Lemma~\ref{lem:extensionbyclosure}, we obtain
\begin{equation}
\label{eq:decomposition3}
\pi_{\X_{k,3} \upharpoonright Y_k}(t) \leq \pi_{\X_{k,2} \upharpoonright Y_k}(t) \text{ for all }t \geq 1.
\end{equation}
Finally, set $I_k = I_{k,3} \setminus \{c_k\}$ and $\X_k[I_k] = \X_{k, 3}[I_k]$.
By Lemma~\ref{lem:vcsubfamily} and \eqref{eq:decomposition1}--\eqref{eq:decomposition3}, we have
\begin{equation}
\label{eq:decomposition4}
\pi_{\X_k \upharpoonright Y_k}(t) \leq \pi_{\X_{k} \upharpoonright Y_{k-1}}(r_k t) \text{ for all }t \geq 1.
\end{equation}
To complete the proof, observe that \eqref{eq:decompositiont} follows by induction from \eqref{eq:decomposition4}.
\end{proof}

\subsection{Deconstruction}
\label{sec:deconstruction}

Throughout this section, fix $i_0 \in I$. Suppose that the uniform closure of $\X[i_0]$ is the flat family $\Z = (Z_{b})_{b \in Y}$ and that $\Upsilon = \{ \h_1, \dots, \h_\ell \}$, with  $\h_j = (H_{j,b})_{b\in Y}$ for each $j \in [\ell]$,  is a uniform essential approximation of $\X[{i_0}]$. For each $b \in Y$, let $\E_b = \E_b := \{ H_{j,b} : j \in [\ell] \}$ be the essential approximation of $X_b$ given by $\Upsilon$ and let $\P_b = \P(\E_b) = \{P_{J, b} : J \subseteq [\ell] \}$ be the partition of $Z_b$ induced by $\E_b$.

Given a family of semilinear sets, the following definition provides a simpler family in a suitable composite space made up from flats, half-flats, and semilinear sets of lower dimension. Take the example $(X_{b_1, b_2})_{(b_1, b_2)\in \RR^2}$ where $X_b$ is the union $P_{b_1, b_2} \sqcup Q_{b_1, b_2}$ of  $P_{b_1, b_2} = \{ (b_1-1, b_2-1)\}$ and $Q_{b_1,b_2}=\{(x_1, x_2) \in \RR^2 : x_1 > b_1, x_2>b_2 \}\setminus\{(b_1+1, b_2+1)  \}$. 

Note that the flat closure of $X_{b_1, b_2}$ is just $\RR^2$ and the essential approximation consists of two half-flats given by $H_{b_1} := \{x_1 > b_1\}$ and $H_{b_2} := \{ x_2 > b_2 \} $. Now, the decomposition of the family $(X_{b_1, b_2})_{(b_1, b_2) \in \RR^2}$ will consist of the half-flat families $(H_{b_1})_{(b_1, b_2) \in \RR^2}$ and $(H_{b_2})_{(b_1, b_2) \in \RR^2}$. It will also contain the constant family all of whose members are $\RR^2$, corresponding to the flat closure. Additionally, for each of the four quadrant defined by  $H_{b_1}$ and $H_{b_2}$, we will have a family whose members are the lower dimensional part of $X_{b_1, b_2}$ in that quadrant (the existence of which is guaranteed by Lemma~\ref{lem:essentialapproxdim}\ref{itm:essentialapproxdim2}).

\begin{defn}[Deconstruction]
\label{def:deconstruction}
We say that $(I', V', \X')$ with $\X' =(X'_b)_{b\in Y}$ is a \emph{deconstruction} of $(I, V, \X)$ at index $i_0$ if the following hold:
\begin{enumerate}[label = (\alph*)]
    \item $I' = (I \setminus \{{i_0}\})  \bigsqcup \left(\{{i_0}\} \times [\ell + 1]\right) \bigsqcup \left(\{{i_0}\} \times 2^{[\ell]}\right)$;
    \item $\X'[I \setminus \{{i_0}\}] = \X[I \setminus \{{i_0}\}]$;
    \item $V'[({i_0},k)] \heq V[i_0]$ for every $j \in [\ell + 1] \cup 2^{[\ell]}$;
    \item $\X'[(i_0, j)] = \h_j$ for every $j \in [\ell]$ and $\X'[(i_0, \ell + 1)] = \Z$;
    \item for every $J \subseteq [\ell]$, we have $\X'[(i_0, J)] = (X_{J, b})_{b \in Y}$ where, for every $b \in Y$,
$$ X_{J, b} \heq
\begin{cases}
	P_{J, b} \cap X_b[i_0] & \text{ if }\dim(P_{J, b} \cap X_b[i_0]) < \dim(X_b[i_0]);\\
	P_{J, b} \setminus X_b[i_0] & \text{ if }\dim(P_{J, b} \cap X_b[i_0]) = \dim(X_b[i_0]).\\
\end{cases}
$$
\end{enumerate}
\end{defn}
The main point of this section is the following.
\begin{prop} \label{prop:deconstruction}
Let $(I', V', \X')$ be a deconstruction of $\X$ at index $i_0 \in I$. Then the following holds:
\begin{enumerate}[label = (\alph*)]
    \item\label{itm:deconstructiona} there is an $r_1 \in \NN$ such that $\pi_\X(t) \leq \pi_{\X'}(r_1t)$ for every $t \geq 1$;
    \item\label{itm:deconstructionb} there is an $r_2 \in \NN$ and a semilinear open cell $U \subseteq Y$ such that $\pi_{\X'\upharpoonright U}(t) \leq \pi_{\X\upharpoonright U}(r_2 t)$ for every $t \geq 1$.
\end{enumerate}
\end{prop}

Given Lemma~\ref{lem:familyseparation2} below, the proof of Proposition~\ref{prop:deconstruction} is essentially the same as that of Proposition~\ref{prop:decomposition}.

\begin{lem} \label{lem:singleseparation2}
Suppose $X \subseteq R^m$ is semilinear and that $Z = \fcl(X)$ is a flat with $\dim(Z) = d > 0$. Let $H \subseteq Z$ be a half-flat which is essential for $X$ and $H_1, \ldots, H_k \subseteq Z$ be half-flats such that any two half-flats in $H, H_1, \dots, H_k$ have distinct boundaries.

Then there is an open cube $C \subseteq \R^m$ such that 
$$ C\cap X = C \cap H \neq \emptyset $$
and, for every $i \in [k]$,
$$ C \cap H_i =C  \quad\text{ or }\quad   C  \cap H_i =\emptyset. $$
\end{lem}
\begin{proof}
By Lemma~\ref{lem:witnesssetdim}, the set $W$ of points witnessing that $H$ is essential for $X$ has dimension $d-1$. Since the flats $H_1, \dots, H_k$ have essential boundaries distinct from the essential boundary of $H$, we may choose a point $a \in W$ such that $a \notin \esb(H_i)$ for every $i\in [k]$.

By definition, there is an open cube $C_0$ centered at $a$ such that $C_0 \cap  X=  C_0 \cap H$. On the other hand, for every $j \in [\ell]$, since $a \notin \esb(H_j)$ there exists a cube $C_j$ centered at $a$ such that $C_j \cap H_j = C$ or $C_j \cap H_j =\emptyset.$ Now any open cube contained in $\bigcap_{i=0}^k C_i$ satisfies the assertions of the lemma.
\end{proof}

The following lemma characterizes the condition that $H$ is essential for $X$ in term of dimension. This is the key to obtain a version of Lemma~\ref{lem:singleseparation2} for families.

\begin{lem}
\label{lem:essentialcriterion}
Let $X \subseteq \R^m$ be a semilinear set and suppose that $Z = \fcl(X)$ is a flat with $\dim(Z) = d > 0$.  Let $H \subseteq Z$ be a closed half-flat such that 
$$\dim ( X \cap \bd(H) \cap \esb(X)  \cap \esb(X \cap H)) = d-1. $$ Then $H$ is essential for $X$.
\end{lem}
\begin{proof}
Let $E = X \cap \bd(H) \cap \esb(X) \cap \esb(X \cap H) $. Note that $\fcl(\bd(X) \setminus \bd(H)))$ is a union of flats that does not contain $\bd(H)$, while $E \subseteq \bd(H)$. It follows that $\dim(E \cap \fcl(\bd(X) \setminus \bd(H)) ) \leq d - 2$. Hence, by replacing $E$ with $E \setminus \fcl((\bd(X)\setminus \bd(H)))$  if necessary, we may assume that $E \cap \fcl((\bd(X)\setminus \bd(H))) = \emptyset$.

Let $a \in \mathrm{int}_{\bd(H)}(E)$ and note that, by definition, there is an open cube $C_1 \subseteq R^m$ centered at $a$ such that $C_1 \cap \bd(H) = C_1 \cap E$. Furthermore, since $\fcl(\bd(X) \setminus \bd(H))$ is a union of flats and $a \notin \fcl(\bd(X) \setminus \bd(H))$ by assumption, there is an open cube $C_2 \subseteq R^m$ centered at $a$ such that $C_2 \cap \fcl(\bd(X) \setminus \bd(H))= \emptyset$. In particular, we have $C_2 \cap (\bd(X)\setminus \bd(H)) = \emptyset.$
Now $C = C_1 \cap C_2$ is an open cube centered in $E$ such that
$$    C \cap \bd(H) = C \cap E \quad\text{ and }\quad C \cap (\bd(X)\setminus \bd(H)) = \emptyset. $$
To show that $H$ is essential for $X$, it suffices to show that $X \cap C = H \cap C$ since this implies that $a$ is a witness that $H$ is essential for $X$.

Let $U_1 = H \setminus \bd(H)$ and $U_2 = R^m \setminus H$ be the two open half-flats in $Z$ determined by $\bd(H)$.
Since $C\cap (\bd(X)\setminus \bd(H)) = \emptyset$, for each $U \in \{U_1, U_2\}$, we have
$$ C \cap (X \cap U) = C \cap U \quad\text{ or }\quad C \cap (X \cap U) = \emptyset.$$

Recalling that $C$ is centered in $E \subseteq \esb(X \cap H)$, we have $\dim(C \cap X \cap H) = d$, implying that $\dim (C \cap (X \cap U_1)) = d$, and, in particular, $C \cap (X \cap U_1) = C \cap U_1$. Now, since $E \subseteq \esb(X)$, it must have non-empty intersection with $Z \setminus X$ implying that $C \cap (X \cap U_2) = \emptyset$.

Finally, it remains to show $C \cap \bd(H) = C \cap (X \cap \bd(H))$. Clearly, $C \cap (X \cap \bd(H)) \subseteq C \cap \bd(H)$. On the other hand, by definition of $E$, we have $C \cap E \subseteq C \cap (X \cap \bd(H))$. Recalling that $C \cap E = C \cap \bd(H)$, we obtain the desired assertion.
\end{proof}

We now prove the promised version of Lemma~\ref{lem:singleseparation2} for families.

\begin{lem}[Second Window Lemma] \label{lem:familyseparation2}
Let $Y \subseteq \R^n$ be an open set and $\X = (X_b)_{b \in Y}$ be a semilinear family in $R^m$. Suppose that the uniform closure of $\X$ is the flat family $\Z = (Z_b)_{b \in Y}$ and that $\h = (H_b)_{b \in Y}$ is a half-flat family such that $H_b$ is essential for $X_b$ for every $b \in Y$. Let $k \in \NN$ and, for every $i \in [k]$, $\h_i = (H_{i,b})_{b \in Y}$ be a half-flat family. Suppose that, for every $b \in Y$, any two half-flats in $H_b, H_{1, b}, \dots, H_{k, b}$ have distinct boundaries.

Then there is an open cube $C \subseteq R^m$ and a semilinear open cell $U \subseteq Y$ such that, for every $b \in U$,
$$ C \cap X_{b}  = C \cap H_b \neq \emptyset$$
and, for every $i \in [k]$,
$$ C \cap H_{i,b}  = C \cap Z_b \quad\text{ or }\quad C \cap H_{i,b}  = \emptyset.  $$

\end{lem}
\begin{proof}
We assume that $H_b$ is closed for every $b \in Y$. This is without loss of generality, since otherwise, we may consider the families $(Z_b \setminus X_b)_{b\in U}$ and $(Z_b \setminus H_b)_{b \in Y}$ instead of $\X$ and $\h$ respectively.
In particular, this allows us to assume $\esb(H_b) \subseteq H_b$ for every $b\in Y$.

Let $d > 0$ be such that $\dim(X_b) = d$ for each $b \in Y$. By Lemma~\ref{lem:witnesssetdim}, for each $b \in Y$, the set $W_b$ of points witnessing that $H_b$ is essential for $X_b$ has dimension $d-1$. Let $b \in Y$ and $\P_b = \{P_{J, b} : J \subseteq [k]\}$ be the partition of $H_b$ induced by the half-flats in $\{H_{1, b}, \dots, H_{k, b}\}$. In particular, let
$$  P_{J,b} = H_b \cap \left( \bigcap_{i \in J} H_{i,b} \right) \cap  \left( \bigcap_{i \in [k] \setminus J} Z_b \setminus H_{i,b} \right).  $$

By Lemma~\ref{lem:essentialapproxdim}\ref{itm:essentialapproxdim1}, for each $J \subseteq [k]$, either $\dim(P_{J,b}) = d$ or $\dim(P_{J,b}) \leq d-2$. Let $H'_b = \bigcup_J P_{J,b}$ with $J$ ranging over the subsets of $[k]$ such that $\dim(P_{J,b}) = d$. 
Then $\dim  (H_{b} \setminus H'_b) \leq d-2 $, and, hence, for each $b \in Y$,
$$\dim ( W_b \cap H'_b ) = d-1.$$
It follows that, for every $b \in Y$, there is $J \subseteq [k]$ such that 
\begin{equation}
\label{eq:familysep2eq1}
\dim(P_{J,b}) = d \quad\text{ and } \quad  \dim ( W_b \cap P_{J,b}) = d-1.
\end{equation}
For every $J \subseteq[k]$, let $Y_J \subseteq Y$ be the set consisting of $b\in Y$ such that \eqref{eq:familysep2eq1} holds. Since this gives a partition of $Y$, there is a $J_0$ such that $\dim(Y_{J_0}) = n$. Let $Y_0 \subseteq Y_{J_0}$ be an open cube.

Now, let $X = \{(a, b) \in \R^m \times \R^n: a \in X_b, b \in Y_0\}$ and let $Z \subseteq \R^m \times \R^n$ be the flat closure of $\{(a, b) \in \R^m \times \R^n: a \in Z_b, b \in Y_0\}$, so $Z$ is the total flat of $\Z$. Since $X \subseteq Z$ and $\dim(X) = \dim(Z) = d+n$ (by Fact~\ref{fact:dimension}-(c)), we have $Z = \fcl(X)$ by Lemma~\ref{lem:affineclosure}. 

Let $H \subseteq Z$ be the total half-flat of $\h$, i.e., $H$ is the half-flat such that, $H_b = \{a \in \R^m: (a, b) \in H\}$ for each $b \in Y$. Then $H$ is an essential half-flat for $X$.
To see this, consider $W = \{ (a, b) \in \R^m \times \R^n : a \in W_b, b \in Y_0 \} \subseteq X \cap \bd_Z(H)$ and let $W' = \int_{\bd(H)}(W)$.  By dimension counting, we have
$$ W' \subseteq X \cap \bd_Z(H) \cap \esb_Z(X) \cap \esb_Z(X \cap H), $$
has dimension $d-1$ implying, by Lemma~\ref{lem:essentialcriterion}, that $H$ is an essential half-flat of $X$.

For each $j \in [k]$, let $H_j \subseteq \R^m \times \R^n$ be the total half-flat of $\h_j$. That is, for each $j \in [k]$ and $b \in Y$, we have $H_{j, b} = \{ a \in \R^m: (a, b) \in H_j\}$. 

Now, applying Lemma~\ref{lem:singleseparation2} to the flats $X, Z, H$, and the collection $\{H_1, \dots, H_k\}$, we obtain a cube $C \times U \subseteq R^m \times R^n$ such that 
\begin{equation}
\label{eq:familyseparation2cubeintersection}
(C \times U) \cap X = (C \times U) \cap H \neq \emptyset
\end{equation}
and, for every $j \in [k]$,
\begin{equation}
\label{eq:familyseparation2cubeintersection2}
(C \times U) \cap H_j = (C \times U) \cap Z  \quad\text{ or }\quad  (C \times U) \cap H_j = \emptyset.
\end{equation}
First note that \eqref{eq:familyseparation2cubeintersection} implies $U \cap Y_0 \neq \emptyset$. By replacing $U$ with $U \cap Y_0$, we may assume that $U$ is an open cube contained in $Y_0$. Hence, we obtain, for every $b \in U$, that $C \cap X_b = C \cap H_b \neq \emptyset$. The remaining assertions follow from \eqref{eq:familyseparation2cubeintersection2}.
\end{proof}

\section{Proof of the main theorems}
\label{sec:mainproof}

In this section, we will prove Theorem~\ref{thm:main} and Theorem~\ref{thm:mainversion2} in a more general form. Throughout, let $\cR =( R; +, <, \ldots)$ be an o-minimal expansion of an ordered Abelian group. 


To prove Theorem~\ref{thm:main}, recall from Section~\ref{sec:proofoverview} that our goal is to construct a simple family whose shatter function is asymptotic to a given semilinear family $\X$. To achieve this, we first construct a family whose component families are flat or half-flat families. We then rely on the following proposition.
\begin{lem}
\label{lem:flattosimple} Suppose $\cR =( R; +, <, (\cdot \lambda)_{\lambda \in D})$ is an ordered vector space over an ordered division ring $D$.
Let $\X$ be a semilinear family in the composite space $V$ indexed by $I$. Supposed that, for every $i \in I$, $\X[i]$ is a flat or half-flat family. Then there is a simple family $\X'$ such that $\pi_{\X'} = \pi_\X$.
\end{lem}
\begin{proof}
For each $i \in I$, by applying a suitable rotation to $\X[i]$, we can arrange that either $\X[i]$ is a family of $d_i$-dimensional flats each parallel to the flat $R^{d_i} \times \{0\}^{m_i-d_i}$, or $\X[i]$ is a family of $d_i$-dimension half-flats each parallel to the flat $R^{d_i} \times \{0\}^{m_i-d_i}$ and with boundary parallel to the $(d_i-1)$-dimension flat  $R^{d_i-1} \times \{0\}^{m_i-d_i+1}$.

Let $V'$ be a composite space indexed by $I$ such that $V[i] \heq R^{m_i-d_i}$ if $\X[i]$ is a family of flat and $V[i] \heq R^{m_i-d_i+1}$ if $\X[i]$ is a family of half-flats. Let $\X'$ be family in composite space $V'$ such that $\X'[i]$ is obtained from $\X[i]$ via a projection to the last $m_i - d_i$ or $m_i - d_i + 1$ coordinates as appropriate.

We claim that $\pi_{\X'} = \pi_\X$. To see $\pi_{\X'}(t) \leq  \pi_\X(t)$, consider a point set $A \subseteq V[I]$ with $|A|=t$ and let $A'$ be the set obtained from $A$ by projections as in the preceding paragraph. Observe that $|A'|\leq t$ and that $|\X' \cap A'|= |\X \cap A| \leq \pi_\X(t)$.

To see that $\pi_{\X'}(t) \geq  \pi_\X(t)$, consider a point set $A' \subseteq V'[I]$ with $|A'|=t$. Let $A$ be the set obtained from $A'$ by setting the last $d_i$ or $d_i + 1$ coordinates to $0$ as appropriate. We now have $|A|= t$ and $|\X \cap A| = |\X' \cap A'| \leq \pi_{\X'}(t)$.
\end{proof}

A key property we rely on is that for flat and half-flat families restricting to a generic parameter set does not change the shatter function.

\begin{lem}
\label{lem:flatfamilyrestiction}
Suppose $\cR =( R; +, <, (\cdot \lambda)_{\lambda \in D})$ is an ordered vector space over an ordered division ring $D$.
Let $\X[I]$ be a semilinear family indexed by a seminlinear open cell $Y \subseteq R^n$ such that for every $i \in I$, $\X[i]$ is a flat or half-flat family. Then, for any semilinear  open $U \subseteq Y$, we have $\pi_{\X\upharpoonright U} = \pi_\X$.
\end{lem}
\begin{proof}
That $\pi_{\X\upharpoonright U}(t) \leq \pi_\X(t)$ is immediate from the definitions. We now prove the other direction.
By replacing $Y$ by a larger index set and $U$ by a smaller index set we can assume $Y= R^n$ and $U= [-\varepsilon, \varepsilon]^n$ with $\varepsilon \in R$. Using affine transformations, we can arrange that in each $V[i]$, the family of flats is given by $x_1=f_1(y), x_k=f_k(y), x_{k+1}= 0, \ldots, x_{m_i}=0$ with each $f_i$ a linear map.

Let $A\subseteq V$ with $|A|=t$ such that $|\X \cap A| = \pi_\X(t) = N$. 
Let $b_1, \dots, b_N \in R^n$ be such that $|\X \cap A| = \{ X_{b_1} \cap A, \dots, X_{b_1} \cap A\}$. Let $K$ be such that $b_1, \ldots, b_N \in [-K, K]^n $. Now scale $A[i]$ by $\varepsilon/K$ to get $A'[i]$ and set $A' = \sqcup_i A[i]$. It is straightforward to verify that $A' \cap X_{\varepsilon b_1/K}, \ldots, A' \cap X_{\varepsilon b_N/K}$ are all distinct. This gives us the desired conclusion.
\end{proof}

The following proposition uses the results in Sections~\ref{sec: Decomdecon} to show that the shatter function of a uniform family must be asymptotic to a polynomial. 
\begin{prop}\label{prop:main2}
Suppose $\cR =( R; +, <, (\cdot \lambda)_{\lambda \in D})$ is an ordered vector space over an ordered division ring $D$.
Let $\X$ be a family in the composite space $V$ indexed by $I$. Suppose that, for every $i \in I$, the family $\X[i]$ is uniform. Suppose also that $\X$ is parameterized by a semilinear open cell $Y \subseteq \RR^n$. Then there is $s \in \NN$ such that, for any open $U \subseteq Y$, we have $\pi_{\X\upharpoonright U}(t) = \Theta(t^s)$.
\end{prop}

\begin{proof}
As mentioned in Section~\ref{sec:proofoverview}, the idea here is to construct a simple family $\X'$ whose shatter function is asymptotic to the shatter function of $\X$.

For this purpose, we first need a notion of complexity. For $i \in I$, the complexity of $(I, V, \X)$ at $i$ is set to be $(0, 0)$ if $\X[i]$ is a flat family or a half-flat family. Otherwise, we set the complexity of $(I, V, \X)$ at $i$ to be the pair  $(\alpha_i, \beta_i) \in \NN^2$ defined as follows. Suppose $\X[i]$ has uniform closure $\Z_1, \dots, \Z_\ell$. Then $\alpha_i$ is the maximum, over $k \in [\ell]$, of the dimension of the members of $\Z_k$ (of course, all of which have the same dimension) and $\beta_i = \ell$. Note that this definition is slightly ad-hoc, in the sense that the complexity $(0, 0)$ is exceptional.

The complexity of $(I, V, \X)$ is the triple $(\alpha, \beta, \gamma)$ where $(\alpha, \beta) \in \NN^2$ is the lexicographic maximal $(\alpha_i, \beta_i)$ as $i$ ranges over $I$, and $\gamma$ is the number of indices $i \in I$ such that $(\alpha_i, \beta_i)$ is lexicographically maximal. 

We now build a finite sequence $(I_k, V_k, \X_k)$  inductively as follows. Set $(I_0, V_0, \X_0) := (I, V, \X)$. Suppose we have constructed $(I_k, V_k, \X_k)$ and that the complexity of $(I_k, V_k, \X_k)$ is $(\alpha_k, \beta_k, \gamma_k)$. If $(\alpha_k, \beta_k, \gamma_k) = (0,0, |I|)$ (i.e., for each $i\in I$, the family $\X[i]$ is a flat or half-flat family), the sequences terminates. Otherwise, we construct $(I_k, V_k, \X_k)$ as follows.

Choose $i \in I_k$ such that the complexity of $\X[i]$ is $(\alpha_k, \beta_k)$. Suppose $\beta_k > 1$ (i.e., $\X_k[i]$ has a uniform closure with $\beta_k$ families) then we set $(I_{k+1}, V_{k+1}, \X_{k+1})$ to be the decomposition of $(I_k, V_k, \X_k)$ at index $i$ (as in Definition~\ref{def:decomposition}).  If $\beta_k = 1$, then we set $(I_{k+1}, V_{k+1}, \X_{k+1})$ to be a deconstruction of $(I_k, V_k, \X_k)$ at index $i$ (as in Definition~\ref{def:deconstruction}).

If $(I_{k+1}, V_{k+1}, \X_{k+1})$ was obtained via a decomposition, then we replaced $\X_k[i]$ with families that have a single family as uniform closure, families of flats, and a family whose members have strictly smaller dimension. If $(I_{k+1}, V_{k+1}, \X_{k+1})$ was obtained via a deconstruction, then we replaced $\X_k[i]$ with flat or half-flat families, and families whose members have strictly smaller dimension. In either case, the complexity of $(I_{k+1}, V_{k+1}, \X_{k+1})$ is lexicographically smaller than the complexity of $(I_{k}, V_{k}, \X_{k})$ and the sequence eventually terminates.

Suppose the sequence terminates after $m$ steps and let $(I_m, V_m, \X_m)$ be the resulting triple. For each $i \in I$, $\X_m[i]$ is a flat or half-flat family.
By Lemma~\ref{lem:flattosimple}, there is a simple family $\X'$ such that $\pi_{\X'}(t) = \pi_{\X_m}(t)$ for all $t \geq 1$. Furthermore, $\pi_{\X'}(t) = \Theta(t^s)$ for some $s \in \NN$ by Theorem~\ref{thm:simpleshatter}. Hence, to prove the theorem it suffices to show that there exist $p, q \in \NN$ such that
\begin{equation}
\label{eq:main2shatterfunctionbound}
\pi_{\X_m}(t) \leq \pi_{\X_0}(p t)\quad\text{ and }\quad\pi_{\X_0}(t) \leq \pi_{\X_m}(q t).
\end{equation}

Note first that, for every $k = 1, \dots, m$, there is $r_k \in \NN$ such that $\pi_{\X_{k-1}}(t) \leq \pi_{\X_k}(r_kt)$. 
This follows from Propositions~\ref{prop:decomposition}\ref{itm:decompositiona} and \ref{prop:deconstruction}\ref{itm:deconstructiona}. Combined with the fact that the number of steps depends only on the family $\X$, we obtain
$$ \pi_{\X_0}(t) \leq \pi_{\X_m}(qt)\text{ where }q = \prod_{k=1}^{m} r_k \in \NN.$$
To see the other direction, let $U_0 = U$, and note that, by Propositions~\ref{prop:decomposition}\ref{itm:decompositionb} and \ref{prop:deconstruction}\ref{itm:deconstructionb}, for every $k = 1, \dots, m$, there exist $s_k \in \NN$ and a semilinear open cell $U_k \subseteq U_{k-1}$ such that $\pi_{\X_{k}\upharpoonright U_{k}}(t) \leq \pi_{\X_{k-1}\upharpoonright U_{k}}(s_k t)$. That $\pi_{\X_{k-1}\upharpoonright U_{k}}(t) \leq \pi_{\X_{k-1}\upharpoonright U_{k-1}}(t)$ follows by definition. Hence, we have
$\pi_{\X_{k}\upharpoonright U_{k}}(t) \leq \pi_{\X_{k-1}\upharpoonright U_{k-1}}(s_kt)$ for each $k \in 1, \dots, m$.
By a repeated application of Lemma~\ref{lem:flatfamilyrestiction}, we have $\pi_{\X'}(t) = \pi_{\X'\upharpoonright U_{m}}(t)$. It follows that
$$\pi_{\X_m}(t) \leq \pi_{\X_0}(pt)\text{ where }p = \prod_{k=1}^{m} s_k \in \NN,$$
which completes the proof.
\end{proof}

We require the following fact from~\cite{jl10} (see also \cite{adhms16}*{Theorem~1.1}).

\begin{fact}[\cite{jl10}*{Lemma~2.2}] \label{fact: vcbound}
    Let $(R; \ldots)$ be an o-minimal structure and let $\S$ be a semilinear set system on $R^m$ parametrized by $Y \subseteq R^n$. Then $\vc(\S) \leq n$.
\end{fact}

Finally, Theorem~\ref{thm:main} is the special case of Theorem~\ref{thm:mainversion1general} below where $D=R = \RR$.

\begin{thm} \label{thm:mainversion1general}
Let $\mathfrak{R} = (R; +, <, (\cdot \lambda)_{\lambda \in D} )$ be an ordered vector space over a division ring $D$ where $(R; <)$ is dense.
Let $\S$ be a semilinear set system on $R$ parameterized by $Y \subseteq R^n$. Then there exist constants $C_1, C_2 > 0$ and an integer $0 \leq s \leq n$ such that
\[ C_1 t^{s} < \pi_\S(t) < C_2 t^{s} \quad \text{for every } t \geq 1. \]
That is, $\pi_\S(t) = \Theta(t^{s})$ (as $t \rightarrow \infty$).
\end{thm}

\begin{proof}
Let $\X = (X_b)_{b\in Y}$ be the semilinear family where $\S(\X)= \S$. We need to show $\pi_\X$ is assymptotic to a polynomial. We note that when $\{Y_1, \ldots Y_k \}$ is a partition of $Y$ and $\X_i = \X\upharpoonright Y_i$ is the restricted family for $i \in [k]$, then 
$$ \pi_{\X}(t) = \Theta\left( \sum_{i \in [k]} \pi_{\X}(t) \right).  $$
Hence, using Proposition~\ref{prop:uniformization}, we can assume that $\X$ is uniform. By Fact~\ref{fact:semilinearfiniteuntion} (semilinear cell decomposition) and the definition of semilinear cells, we can further arrange that $Y \subseteq \RR^n$ is an open cell. By Proposition~\ref{prop:main2}, $\pi_\S(t)$ is asymptotic to a polynomial. The desired conclusion follows from Fact~\ref{fact: vcbound}.
\end{proof}


Toward proving a generalization of Theorem~\ref{thm:mainversion2}, we need the following fact.

\begin{fact}[\cite{adhms16}, Proposition 4.6] \label{fac: vcnonint}
    There is a definable family $\X$ in $\cR$ such that $\pi_\X(t)  = \Theta( t^{4/3})$.
\end{fact}

Theorem~\ref{thm:mainversion2} is the special case of Theorem~\ref{thm:mainversion2general} below with $R=\RR$.

\begin{thm} \label{thm:mainversion2general}
Let $\mathfrak{R}= (R; +, <, \ldots)$ be an o-minimal structure expanding a divisible ordered abelian group. Then the following are equivalent:
\begin{enumerate}[label = (\alph*)]
    \item $(R; +, <, \ldots)$ is weakly locally modular;
    \item for any set system $\S$ definable in $(\RR; \ldots)$, $\pi_\S(t)$ is assymptotic to a polynomial.
\end{enumerate}
\end{thm}

\begin{proof}

We first prove the forward implication. Suppose $(R; +, <, \ldots)$ is weakly locally modular. Then it is linear by Zilber's trichotomy (Fact~\ref{fac: o-min trich}). Hence, by passing to an elementary extension and Fact~\ref{fact: Onebasedorderstructure}), we can assume that $\cR$ is an ordered vector space of a potentially larger ordered division ring $D$. Then the desired conclusion follows from Theorem~\ref{thm:mainversion1general}.

Now we prove the backward implication. Suppose $(R; +, <, \ldots)$ is not weakly locally modular. Then by Zilber's trichotomy (Fact~\ref{fac: o-min trich}), there is a real closed field definable in $(R; +, <, \ldots)$. Hence, by Fact~\ref{fac: vcnonint}, one can get a definable family with vc-dimension $4/3$, so its shatter function is not assymptotic to a polynomial.
\end{proof}

One can show the same result for other o-minimal structures by relating it to structures in~\ref{thm:mainversion2general}. Below is an example.

\begin{cor}
\label{cor:multversion}
Suppose $\cR =(R; \times, <, \ldots)$ is an o-minimal structure expanding
a model of $\mathrm{Th}(\RR; \times, <)$ 
\begin{enumerate}[label = (\alph*)]
    \item $(R; \ldots)$ is weakly locally modular;
    \item for any set system $\S$ definable in $(\RR; \ldots)$, $\pi_\S(t)$ is assymptotic to a polynomial.
\end{enumerate}
\end{cor}
\begin{proof}
We note that $(R; +, \times)$ is an ordered abelian group, and $\mathrm{Th}(\RR; \times, <, \ldots)$ is interdefinable in an expansion of the ordered abelian group $(R^{>0}; \times, <)$. Hence, both cases of this statement follows from the preceding theorem.
\end{proof}

\bibliographystyle{abbrv}
\bibliography{references}

\appendix

\section{Model theoretic preliminaries}
\label{app:modelprelim}

Throughout this section, let $\cR= (R; <, 0, 1, +, (\lambda\cdot)_{\lambda \in D})$ be a vector space over an ordered division ring $D$. We view $\cR$ as a structure in the language $L = \{<, 0, 1, +, (\lambda\cdot)_{\lambda \in D}\}$.

\subsection{Linear inequalities and semilinear formulas} \label{App:A1}

 The material in this section is based on~\cite{vandries1998}*{Section~1.7}, where $D$ is taken to be an ordered field. However, the results remain valid for ordered division rings, as noted in Remark~\cite{vandries1998}*{Section~1.7.10}.

A linear polynomial in $\cR$ (or, more precisely, in $L$) is an $L$-term $\lambda(\bar{x})$ with $\bar{x}=(x_1, \ldots, x_m)$ of the form $\lambda_1x_1+ \ldots+ \lambda_m x_m$ with with $\lambda_1, \dots, \lambda_m \in D$. As usual, we will write $x \leq y$ for $(x<y) \vee (x=y)$. A \emph{linear inequality in $\cR$} is an $L(R)$-formula $\phi(\bar{x})$ of the form
 $$\lambda(\bar{x}) <  c  \quad\text{or} \quad  \lambda(\bar{x}) \leq   c $$
with $\lambda(\bar{x})$ a linear polynomial and $c \in R$. If $c = \lambda'_1 b_1 + \dots \lambda'_n b_n$ for some $\lambda'_1, \dots, \lambda'_n \in D$ and $b_1, \ldots, b_n \in B \subseteq R$, then we say $\phi$ is a \emph{linear inequality over $B$}.

A \emph{system of linear inequalities} is a conjunction of linear inequalities and a \emph{semilinear} formula is a disjunction of systems of linear inequalities. If the inequalities in a system or a semilinear formula are over $B$, then we say the system of inequalities or the semilinear formula is over $B$.  A subset of $R^m$ is \emph{semilinear over $B$} if it is the solution set of (i.e., the set defined by) a semilinear formula over $B$.
\begin{fact}[\cite{vandries1998}*{Corollary 1.7.8}] \label{fact:semilineardefinability}
$X \subseteq R^m$ is definable in $\cR$ over $B\subseteq R$ if and only if it is  semilinear over $B$.
\end{fact}

The following is a straightforward consequence of quantifier elimination for the theory of $\cR$ (see, e.g.,~\cite{vandries1998}*{Corollary 1.7.8}).
\begin{fact} \label{fact:linearequationover parameter}
Suppose $X$ is the solution set of a semilinear formula in $\cR$ and that $X$ is definable over $B\subseteq R$ in the sense of model theory. Then $X$ is the solution set of a semilinear formula over $B$.
\end{fact}

A \emph{linear equation in $\cR$} is an $L(R)$-formula $\phi(\bar{x})$ of the form
 $$\lambda(\bar{x}) = c, $$
with $\lambda(\bar{x})$ a linear polynomial. If $c = \lambda'_1 b_1 + \dots \lambda'_n b_n$ for some $\lambda'_1, \dots, \lambda'_n \in D$ and $b_1, \ldots, b_n \in B \subseteq R$, then we say $\phi$ is a \emph{linear equation over $B$}.

A \emph{system of linear equations} in $\cR$ is a conjunction of linear equations in $\cR$ and a \emph{linear formula} is
a disjunction of systems of linear equations in $\cR$. If the equations in a linear formula are over $B$, then we say the system or the linear formula is over $B$.  We say that $X \subseteq R^m$ is a \emph{flat} (in $\cR$) over $B\subseteq R$ if $X$ is the solution set of a system of linear equations over $B$.  Clearly, a flat over $B$ is semilinear over $B$.

Suppose $Z \subseteq R^m$ is a flat. We say that $H \subseteq Z$ is  \emph{half-flat}  if $H \neq \emptyset$, $H\neq Z$, and $H$ is the solution set of a linear inequality $\phi(x_1, \dots, x_m)$. That is, $H$ can be written as $$\{ (a_1, \ldots, a_m) \in Z :  \lambda_1a_1 +\ldots \lambda_ma_m \ \square \ c \}$$
with $\lambda_i \in D$, $ c \in R$, and $ \square \in \{ <, \leq\}$. 
If $Z$ is a flat over $B$ and $\phi$ is a linear equality over $B$, then we say $H$ is a half-flat over $B$.
Clearly, a half-flat $H \subseteq Z$ is open, in the Euclidean topology on $\cR$, if it is defined by a strict linear inequality and closed otherwise.

\subsection{Semilinear cell decompositions}
\label{sec:semilinearcell}
Here, we adapt the definitions of a cell in an o-minimal structure and cell decompositions (see \cite{vandries1998}*{Chapter 3.2}) to our setting. 

For every definable set $X \subseteq R^m$, set
$$L(X) := \left\{ f: X \rightarrow R : f\text{ is linear}\right\} \quad\text{ and }\quad
L_\infty(X) := L(X) \cup \{-\infty, +\infty\},$$
where we regard $-\infty$ and $+\infty$ as constant functions on $X$.

For $f, g \in L_\infty(X)$ we write $f < g$ if $f(x) < g(x)$ for every $x \in X$, and in this case we put
$$(f, g)_X := \{(x,r) \in X \times R: f(x) < r < g(x)\}. $$
So $(f, g)_X$ is a semilinear subset of $R^{m+1}$. 

Let $(i_1, \dots, i_m)$ be a sequence of zeroes and ones of length $m$. An \emph{$(i_1, \dots, i_m)$-cell} is a semilinear subset of $R^m$ obtained by induction on $m$ as follows:
\begin{enumerate}
    \item a null-cell is the one-element set $R^0$, a $(0)$-cell is a one-element set $\{r\} \subseteq R$ (a ``point''), and a $(1)$-cell is an interval $(a, b) \subseteq R$, and $a, b \in R \cup \{-\infty, \infty\}$;
    \item suppose $(i_1, \dots, i_m)$-cells are already defined; then an $(i_1, \dots, i_m, 0)$-cell is the graph $\Gamma(f)$ of a linear function $f \in C(X)$, where $X$ is an $(i_1, \dots, i_m)$-cell; an $(i_1, \dots, i_m, 1)$-cell is  a set $(f, g)_X$, where $X$ is an $(i_1, \dots, i_m)$-cell and $f, g \in C_\infty(X), f < g$.
\end{enumerate}
We will refer to an $(i_1, \dots, i_m)$-cell as a \emph{semilinear cell}. The following is a consequence of the o-minimal cell decomposition theorem along with the fact that a definable function in $\cR$ is piecewise linear (see~\cite{vandries1998}*{Corollary 1.7.6}).
\begin{fact}
\label{fact:semilinearfiniteuntion}
Every semilinear set is a disjoint finite union of semilinear cells.
\end{fact}

Note that a semlinear open cell is simply a $(1, \dots, 1)$-cell.
The following is an immediate consequence of the definitions.
\begin{lem} \label{lem:cubeparameterspace}
Suppose $\X = (X_b)_{b\in Y}$ is a semilinear family and $Y$ is a semilinear cell. Then there is a semilinear family $\X' = (X'_{b})_{b\in Y'}$ with $Y'$ a semilinear open cell such that $\X$ and $\X'$ correspond to the same set system. 
\end{lem}

\subsection{Dimension} \label{App:A3}

Here, we state some standard properties of dimension. For a set $X \subseteq R^m$, we write $\dim_\cR(X)$ for the dimension of $X$ in $\cR$. When $\cR$ is clear from context, we write $\dim(X)$ instead of $\dim_\cR(X)$. 

\begin{fact}
\label{fact:dimension}
Suppose $X$ and $Y$ are semilinear in $\cR$. Then we have the following:
\begin{enumerate}[label=(\alph*)]
    \item $\dim_\cR (\emptyset) = -\infty $; $\dim_\cR (X) =0$  if and only if $X$ is finite; if $I=(a,b)$ with $a,b \in \R \cup\{-\infty, +\infty\}$ and $a<b$, then $\dim_\cR (I)  =1$ (see~\cite{vandries1998}*{page 5}).
    \item $\dim_\cR(X \cup Y) =\max\{ \dim_\cR (X), \dim_\cR (Y)\}$ (see~\cite{vandries1998}*{Proposition~4.1.3}).
    \item If $\X = (X_b)_{b\in Y}$ is a semilinear family and $X = \{(a,b) :a \in X_b\}$, then, for every $d \leq \dim (X)$, we have that $ Y_d := \{ b\in Y : \dim_\cR(X_b) =d \} $    is semilinear, 
    and 
    $$ \dim_\cR (X) = \max_{d \leq \dim (X)} (d + \dim_\cR (Y_d));       $$
    in particular, if $f: X \to Y$ is a definable bijection, then $\dim_\cR(X) = \dim_\cR (Y)$ and  $\dim_\cR (X \times Y ) = \dim_\cR (X) + \dim_\cR (Y)$ (see~\cite{vandries1998}*{Corollary 4.1.6}).
    \item  If $\cR=\RR$ and $X \subseteq \RR^m$, then $\dim_\RR (X) = m$ if and only if $X$ contains an open subset (see~\cite{vandries1998}*{Section 4.1.1}). 
    \item If $\X = (X_b)_{b\in Y}$ is a definable family in $\cR$, then there is ${N \in \NN^{\geq 1}}$ such that either $|X_b|< N$ or $\dim_\cR(X_b) \geq 1$ (see~\cite{vandries1998}*{Lemma~3.1.7 and Corollary 4.1.6}).
\end{enumerate}
\end{fact}

\begin{fact} \label{fact:definabilitydimension}
Suppose $\phi(x, y)$ is a semilinear formula in $\cR$, and $\cR'=(R', \ldots)$ is an elementary extension of $\cR$. Then there is a semilinear formula $\theta$ in $\cR$ independent of the choice of $\cR'$ such that, for every $b\in (R')^{|y|}$, we have 
$$ \dim_{\cR'} (\phi( \cR', b)) =d \text{ if and only if } \cR' \models \theta(b). $$
In particular, for every $b\in R^{|y|}$, we have $\dim_{\cR'} (\phi( \cR', b)) = \dim_{\cR} (\phi( \cR, b)) $.
\end{fact}

Let $X$ be a semilinear subset of $\R^m$ and $a\in \R^m$. We set
$$ \dim_a(X) = \inf \{ \dim_\cR( U \cap X) : U \subseteq R^m \text{ is semilinear and } a\in U\}. $$
As the topology on $R^m$ is the Euclidean topology, in the above definition, we can restrict our attention to open cubes centered at $a$.

The following is a standard consequence of the fact that $(\R; +, <)$ is o-minimal.
\begin{fact} \label{fact:dimandlocaldim}
Suppose $X\subseteq \R^m$ is semilinear. Then $\dim_\cR(X) = \sup_{a \in \R^m} \dim_a(X)$. Moreover, if $Z$ is a flat, then $\dim_a(Z)=\dim_\cR(Z)$ for all $a \in Z$. 
\end{fact}

\subsection{Stability-theoretic results}

Let $\cR =(R; +,<, \ldots)$ be a saturated o-minimal structure expanding a divisible ordered abelian group. For $A \subseteq R$ and $a \in R^m$, we set 
$\dim(a/A) = \min\{ X \subseteq R^m : a \in X, X \text{ is definable over } A\}$. Given a finite tuple $a$ and sets $C,B \subseteq R$, we write $a \ind_C B$ to denote that $\dim \left(a/BC \right) = \dim \left(a / C \right)$.

We say that $\cR$ is \emph{weakly locally modular} if for all  small subsets $A,B$ of $R$, there exists some small set $C \ind_{\emptyset} AB$ such that $A \ind_{\acl(AC) \cap \acl(BC)} B$.

A family of plane curves in $\cR$ is a definable family $(X_b)_{b \in Y}$ of subsets of $R^2$ such that $\dim(X_b)=1$ for all $b \in Y$. Such a family is normal if $\dim(X_b\cap X_{b'})=0$ for all distinct $b, b'\in Y$. An $o$-minimal structure  $\cR$ is \emph{linear} (i.e.~any
normal interpretable family of plane curves in $T$ has dimension $\leq 1$). 

It is known that, for o-minimal structures, these two notion are equivalent.
\begin{fact}[\cite{berenstein2012weakly}*{Section 3.2}]\label{fac: o-min trich}
Let $\mathcal{M}$ be an $o$-minimal ($\aleph_1$-)saturated structure.	The following are equivalent:
\begin{itemize}
	\item $\mathcal{M}$ is not linear;
	\item $\mathcal{M}$ is not weakly locally modular;
	\item there exists a real closed field definable in $\mathcal{M}$.
\end{itemize}
\end{fact}

The following is essentially a consequence of quantifier elimination.
\begin{fact} \label{fac: vectorislinear}
      Every theory of an ordered vector space over an ordered division ring is weakly locally modular.
\end{fact}
A partial converse is given by the following result.
\begin{fact}[\cite{LoPe93}, Theorem~6.1]  \label{fact: Onebasedorderstructure}
Suppose $\mathfrak{R} = (R; +, <, \ldots)$ is an o-minimal structure expanding an ordered abelian group, and $ \mathfrak{R}$ is linear. Then $ \mathfrak{R}$ has an elementary extension which is a reduct of an ordered vector space of an ordered division ring.
\end{fact}

\end{document}